\documentclass[x11names,11pt,twoside]{amsart}

\usepackage[utf8]{inputenc}
\usepackage[english]{babel}
\usepackage[T1]{fontenc}
\usepackage{amsfonts}
\usepackage{geometry}
\usepackage{tikz}
\tikzstyle{every picture}=[line width=.7pt,minimum size=3pt,every label/.append style={font=\normalsize},label distance=2pt]
\tikzstyle{every node}=[font=\normalsize,circle,draw=black,fill=black,inner sep=0pt,minimum width=1.3pt]
\usepackage{amsmath, amsthm, amssymb}
\usepackage{latexsym}
\usepackage{enumitem}
\usepackage{fancyhdr}
\usepackage{mathtools}
\usepackage{caption}
\usepackage[labelformat=simple]{subcaption}

\usepackage{color}
\usepackage{bm}
\usepackage[bookmarks=true]{hyperref}

\makeatletter
\newtheorem*{rep@theorem}{\rep@title}
\newcommand{\newreptheorem}[2]{%
\newenvironment{rep#1}[1]{%
 \def\rep@title{#2 \ref{##1}}%
 \begin{rep@theorem}}%
 {\end{rep@theorem}}}
\makeatother

\theoremstyle{plain}
\newtheorem{theorem}{Theorem}[section]
\newtheorem{proposition}[theorem]{Proposition}
\newtheorem{corollary}[theorem]{Corollary}
\newtheorem{lemma}[theorem]{Lemma}
\newtheorem{conjecture}[theorem]{Conjecture}
\theoremstyle{definition}
\newtheorem{definition}[theorem]{Definition}
\newtheorem{example}[theorem]{Example}
\newtheorem{examples}[theorem]{Examples}

\newtheorem{notation}[theorem]{Notation}
\newtheorem{question}[theorem]{Question}

\newtheorem{problem}[theorem]{Problem}
\newtheorem{remark}[theorem]{Remark}
\newreptheorem{theorem}{Theorem}

\date{}

\newcommand{\hgt}{{\rm ht}}

\newcommand{\pd}{{\rm pd}}

\newcommand{\depth}{{\rm depth}}
\newcommand{\reg}{{\rm reg}}
\newcommand{\chara}{{\rm char}}
\newcommand{\cone}{{\rm cone}}
\newcommand{\mc}{\mathcal}

\title{Cohen-Macaulay binomial edge ideals and\\ accessible graphs}

\author{Davide Bolognini, Antonio Macchia, Francesco Strazzanti}

\address{{\small Davide Bolognini, Dipartimento di Matematica, Universit\`a di Bologna, Piazza di Porta San Donato 5, 40126 Bologna, Italy}}
\email{{\small davide.bolognini.cast@gmail.com}}

\address{{\small Antonio Macchia, Fachbereich Mathematik und Informatik, Freie Universit\"at Berlin, Arnimallee 2, 14195 Berlin, Germany}}
\email{{\small macchia.antonello@gmail.com}}

\address{{\small Francesco Strazzanti, Dipartimento di Matematica ``Giuseppe Peano'', Universit\`a degli Studi di Torino, Via Carlo Alberto 10, 10123 Torino, Italy}}
\email{{\small francesco.strazzanti@gmail.com}}

\begin{document}

\begin{abstract}
The cut sets of a graph are special sets of vertices whose removal disconnects the graph. They are fundamental in the study of binomial edge ideals, since they encode their minimal primary decomposition.
We introduce the class of \textit{accessible graphs} as the graphs with unmixed binomial edge ideal and whose cut sets form an accessible set system. We prove that the graphs whose binomial edge ideal is Cohen-Macaulay are accessible and we conjecture that the converse holds.
We settle the conjecture for large classes of graphs, including chordal and traceable graphs, providing a purely combinatorial description of Cohen-Macaulayness. The key idea in the proof is to show that both properties are equivalent to a further combinatorial condition, which we call \textit{strong unmixedness}.
\end{abstract}

%%%%
\maketitle
%%%%



\section{Introduction}

Binomial edge ideals, introduced in 2010/11 in \cite{HHHKR10} and \cite{O11}, are quadratic binomial ideals associated with finite simple graphs. They are generated by certain $2$-minors of a $(2 \times n)$-generic matrix corresponding to the edges of a graph on $n$ vertices; more precisely, the \textit{binomial edge ideal} of a graph $G$ is the ideal
\[
J_G=(x_i y_j - x_j y_i : \{i,j\} \in E(G)) \subseteq K[x_1,\dots,x_n,y_1,\dots,y_n],
\]
where $E(G)$ is the edge set of $G$ and $K$ is a field. In this sense, they generalize the ideals of $2$-minors and in the last ten years gave rise to a rich and active research avenue. They also arise in the study of conditional independence statements in Algebraic Statistics \cite[Section 4]{HHHKR10} and are a subclass of the so-called Cartwright-Sturmfels ideals \cite[Section 3]{CDG18}.

Exploiting the combinatorics of the underlying graph, many authors have studied algebraic and homological properties and invariants of these ideals, such as their regularity \cite{JKS20, KS16, MM13, RSK20b, RSK20, SK18}, depth \cite{BN17, RSK21}, local cohomology \cite{A19}, universal Gr\"obner basis \cite{BBS17} and licci property \cite{ERT20}. In particular, their primary decomposition and unmixedness can be characterized combinatorially. Indeed, given a graph $G$, the minimal prime ideals of $J_G$ are in bijection with the so-called cut sets of $G$, see \cite[Corollary 3.9]{HHHKR10}. Recall that a \textit{cut set} is a subset $S$ of vertices of $G$ such that either $S=\emptyset$ or $c_G(S\setminus \{s\})<c_G(S)$ for every $s \in S$, where $c_G(S)$ denotes the number of connected components of the graph obtained from $G$ by removing the vertices of $S$. By \cite[Lemma 2.5]{RR14}, $J_G$ is \textit{unmixed} if and only if $c_G(S) = |S| + c$ for every $S \in \mc C(G)$, where $\mc C(G)$ is the collection of cut sets of $G$ and $c$ is the number of connected components of $G$.

In general, it is not easy to determine whether an ideal is Cohen-Macaulay, also due to the limitations of symbolic computations. Therefore, it is very interesting to find alternative descriptions of Cohen-Macaulayness. In this direction, several authors found constructions (\cite{KS15, RR14}) and described classes of graphs whose binomial edge ideal is Cohen-Macaulay (\cite{BMS18, EHH11, R13, R19}). In this paper we present the first attempt to find a general combinatorial characterization of Cohen-Macaulay binomial edge ideals, which is only based on the structure of the cut sets of a graph, providing a simpler way to check such homological property.

In a previous paper, \cite{BMS18}, we give a classification of bipartite graphs with Cohen-Macaulay binomial edge ideal, providing an explicit construction in graph-theoretical terms. In \cite[Theorem 6.1]{BMS18}, we also present a further combinatorial characterization of Cohen-Macaulayness in terms of cut sets: if $G$ is bipartite, then $J_G$ is Cohen-Macaulay if and only if
\[
\begin{array}{cl}
(\ast) & \text{$J_G$ is unmixed and $\mc C(G)$ is an \textit{accessible set system}, i.e., for every non-empty $S \in \mc C(G)$}\\[-2.5mm]
 & \text{there exists $s \in S$ such that $S \setminus \{s\} \in \mc C(G)$.}
\end{array}
\]
\noindent In this paper, we call \textit{accessible} a graph with property $(\ast)$. The previous equivalence and further computational evidence motivated us to formulate the following:

\begin{conjecture}\label{C.mainConj}
Let $G$ be a graph. Then, $J_G$ is Cohen-Macaulay if and only if $G$ is accessible.
\end{conjecture}

A topological characterization of Cohen-Macaulayness has been recently proved by Àlvarez Montaner in \cite{A19}, relating this algebraic property to the vanishing of the reduced cohomology groups of a certain poset arising from the minimal prime ideals of $J_G$. The structure of this poset can be rather complicated even for relatively small graphs. Moreover, from \cite[Corollary 3.11]{A19}, it is not clear whether the Cohen-Macaulayness of $J_G$ depends on the field. On the contrary, Conjecture \ref{C.mainConj} would provide a combinatorial and field-independent characterization in terms of cut sets.

In Section 3 the poset introduced by Àlvarez Montaner turns out to be an important tool to prove one implication of Conjecture \ref{C.mainConj}:

\begin{reptheorem}{T.CMimpliesRCSP}
Let $G$ be a graph. If $J_G$ is Cohen-Macaulay, then $G$ is accessible.
\end{reptheorem}

As a consequence, we show that \cite[Conjecture 1.6]{BV15}, about the diameter of the dual graph of an ideal, holds for all binomial edge ideals, see Corollary \ref{C.Hirsch}.

In Section 4, we start a systematic study of accessible graphs and of their cut sets. In particular, by Theorem \ref{T.CMimpliesRCSP}, the properties of accessible graphs are also properties of graphs with Cohen-Macaulay binomial edge ideal. This gives further combinatorial ways to check whether $J_G$ is not Cohen-Macaulay for a given graph $G$. To state the next result, recall that a vertex $v$ of $G$ is called \textit{cut vertex} if the graph obtained by removing $v$ has more connected components than $G$.

\begin{theorem}
Let $G$ be a connected accessible graph.\\
\indent {\rm [\textbf{Remark \ref{R.NoCutVerticesComplete}}]:} If $G$ has no cut vertices, then it is a complete graph.\\
\indent {\rm [\textbf{Lemma \ref{L.oneCutVertex}}]:} If $G$ has one cut vertex, then it is a cone over two connected accessible graphs with fewer vertices than $G$.\\
\indent {\rm [\textbf{Theorem \ref{T.propertiesCM}}]:} If $G$ has at least two cut vertices, then:
    \begin{itemize}
    \item[{\rm (1)}] every non-empty cut set of $G$ contains a cut vertex;
    \item[{\rm (2)}] the graph induced on the cut vertices of $G$ is connected;
    \item[{\rm (3)}] every vertex of $G$ is adjacent to a cut vertex.
    \end{itemize}
In particular, these properties hold if $J_G$ is Cohen-Macaulay.
\end{theorem}

Along the way, we prove that if $G = {\rm cone}(v, H_1 \sqcup H_2)$ and $J_G$ is Cohen-Macaulay, then $J_{H_1}$ and $J_{H_2}$ are Cohen-Macaulay by Theorem \ref{T.CMcones}, showing the converse of \cite[Theorem 3.8]{RR14}.

To study the other implication of Conjecture \ref{C.mainConj}, we introduce the class of \textit{strongly unmixed} binomial edge ideals, see Definition \ref{D.StronglyUnmixed}. The main result of Section 5, Theorem \ref{T.StronglyUnmixedImpliesCM}, shows that strong unmixedness implies Cohen-Macaulayness. Summarizing, we have:
\[
J_G \text{ strongly unmixed } \Longrightarrow J_G \text{ Cohen-Macaulay } \Longrightarrow G \text{ accessible},
\]
where both strong unmixedness and accessibility are purely combinatorial conditions.

By virtue of Proposition \ref{P.RCSPImpliesStronglyUnmixed} and Corollary \ref{C.rcsp2}, proving Conjecture \ref{C.mainConj} boils down to show that every non-complete accessible graph $G$ has a cut vertex $v$ such that $J_{G \setminus \{v\}}$ is unmixed, see Question \ref{Q.GoodCutVertex}.

In Section 6, we focus on two important classes: \textit{chordal graphs} and graphs containing a Hamiltonian path, called \textit{traceable graphs}. We prove that chordal and traceable graphs, if accessible, have the cut vertex we are looking for. In particular, we show that for these graphs being accessible is equivalent both to $J_G$ Cohen-Macaulay and to $J_G$ strongly unmixed, see Theorems \ref{T.FinalChordal} and \ref{T.traceableCM}. This also shows that the Cohen-Macaulayness of $J_G$ does not depend on the field for chordal and traceable graphs, even if the graded Betti numbers of $J_G$ may depend on the field, as in Example \ref{E.BipartiteFieldDependent}.

In Corollary \ref{C.bipartiteCM} we also notice that accessible bipartite graphs are traceable, thus recovering for these graphs the equivalence between $J_G$ Cohen-Macaulay and $G$ accessible proved in \cite[Theorem 6.1]{BMS18}.

We conclude by discussing some open questions in Section 7.

\section{Preliminaries}

Throughout the paper, all graphs will be finite and simple, i.e., undirected graphs with no loops nor multiple edges. Given a graph $G$, we denote by $V(G)$ and $E(G)$ its vertex and edge set, respectively. For every vertex $v$ of $G$, we denote by $N_G(v) = \{w \in V(G) : \{v,w\} \in E(G)\}$ the set of neighbours of $v$ in $G$ and we set $N_G[v]=N_{G}(v) \cup \{v\}$. Given $W\subseteq V(G)$, the \textit{induced subgraph} by $W$ in $G$ is the graph $G[W]$ with vertex set $W$ and whose edge set consists of the edges of $G$ with both endpoints in $W$.

To simplify the notation, if $S \subseteq V(G)$, we denote by $G \setminus S$ the induced subgraph $G[V(G)\setminus S]$, which is the graph obtained by removing from $G$ the vertices of $S$ and all the edges incident in them. In particular, $G \setminus \{v\}$ denotes the graph obtained by removing the vertex $v$ and all edges containing $v$.

A vertex $v \in V(G)$ is said to be a \textit{cut vertex} or \textit{cut point} of $G$ if $G\setminus \{v\}$ has more connected components than $G$. Given $S \subseteq V(G)$, we denote by $c_G(S)$ or simply $c(S)$ (if the graph is clear from the context) the number of connected components of $G \setminus S$.  Moreover, we say that $S$ is a \textit{cut-point set} or simply \textit{cut set} of $G$ if either $S=\emptyset$ or $c_G(S\setminus \{s\})<c_G(S)$ for every $s \in S$. In particular, the cut sets of cardinality $1$ are the cut vertices of $G$. We denote by $\mc C(G)$ the collection of cut sets of $G$.

In this context, when we say that, given a cut set $S$ of $G$, a vertex $v \in S$ \textit{reconnects} some connected components $G_1,\dots,G_r$ of $G \setminus S$, we mean that if we add back $v$ to $G \setminus S$, together with all edges of $G$ incident in $v$, then $G_1,\dots,G_r$ are in the same connected component.

\medskip
Cut sets are very important in the study of binomial edge ideals because they allow to describe the minimal primary decomposition of $J_G$, as we are going to explain.

Let $G$ be a graph with vertex set $[n]=\{1,\dots,n\}$, $K$ be a field and consider the polynomial ring in $2n$ indeterminates $R=K[x_1,\dots,x_n, y_1,\dots,y_n]$. The \textit{binomial edge ideal} of $G$ is the ideal
\[
J_G=(x_i y_j - x_j y_i : \{i,j\} \in E(G)) \subseteq R.
\]
For every $S \subseteq V(G)$, we set
\[
P_S(G) = (x_i,y_i : i \in S)+J_{\widetilde{G}_1}+\dots+J_{\widetilde{G}_{c(S)}},
\]
where $G_1, \dots, G_{c(S)}$ are the connected components of $G \setminus S$ and  $\widetilde{G}_j$ is the complete graph on the vertex set $V(G_j)$.

By \cite[Section 3]{HHHKR10}, $P_S(G)$ is a prime ideal with height $n-c(S)+|S|$, it contains $J_G$ and it is a minimal prime ideal of $J_G$ if and only if $S$ is a cut set of $G$. Moreover, the minimal primary decomposition of $J_G$ is $J_G=\cap_{S \in \mc C(G)} P_S(G)$.

We recall that, an ideal is \textit{(height-)unmixed} if all its minimal prime ideals have the same height. Thus, since $\emptyset \in \mc C(G)$, it easily follows that $J_G$ is unmixed if and only if $c_G(S)=|S|+c$ for every $S \in \mc C(G)$, where $c$ is the number of connected components of $G$. In this case, $\dim(R/J_G) = n+c$.

\begin{remark}\label{R.NoCutSetsComplete}
For a graph $G$, $\mathcal{C}(G)=\{\emptyset\}$ if and only if the connected components of $G$ are complete graphs. Moreover, if $G$ is connected, then $J_G$ is the ideal of $2$-minors of a $(2 \times n)$-generic matrix and, hence, Cohen-Macaulay, see \cite[Corollary 2.8]{BV88}.
\end{remark}

We introduce a class of graphs whose binomial edge ideal is unmixed, which will be the main object of study in the paper.

\begin{definition}\label{D.recursiveCutSetProp}
A graph $G$ is \textit{accessible} if $J_G$ is unmixed and $\mc C(G)$ is an \textit{accessible set system}, i.e., for every non-empty cut set $S \in \mc C(G)$ there exists $s \in S$ such that $S \setminus \{s\} \in \mc C(G)$.
\end{definition}

Notice that this is a purely combinatorial notion, since unmixedness can also be phrased in terms of the graph.

\begin{example}\label{E.accessible}
The graph $G$ in Figure \ref{F.accessibleGraph} is accessible. In fact, its cut sets are
\begin{align*}
\mc C(G) \! =& \{\emptyset, \{2\}, \{5\}, \{10\}, \{2, 5\},  \{2, 10\}, \{3, 10\}, \{4, 10\},\! \{5, 7\},\! \{5, 10\},\! \{2, 4, 5\},\! \{2, 4, 10\},\! \{2, 5, 7\},\! \{2, 5, 10\}, \\
 & \{4, 5, 10\}, \{5, 7, 10\}, \{2, 4, 5, 7\}, \{2, 4, 5, 10\}, \{2, 5, 7, 10\}, \{3, 5, 7, 10\}, \{4, 5, 7, 10\}, \{2, 4, 5, 7, 10\}\}
\end{align*}
and it is easy to check that $\mc C(G)$ is an accessible set system.

\begin{figure}[ht!]
\begin{subfigure}[c]{0.45\textwidth}
\centering
\begin{tikzpicture}
\node[label={below:{\small $1$}}] (a) at (4,0) {};
\node[label={below:{\small $2$}}] (b) at (2.5,0) {};
\node[label={below:{\small $3$}}] (c) at (1,0) {};
\node[label={below right:{\small $4$}}] (d) at (1.5,0.75) {};
\node[label={below:{\small $5$}}] (e) at (0,0.75) {};
\node[label={below:{\small $6$}}] (f) at (-1,0) {};
\node[label={left:{\small $7$}}] (g) at (-1.5,0.75) {};
\node[label={above:{\small $8$}}] (h) at (-1,1.5) {};
\node[label={above:{\small $9$}}] (i) at (1,1.5) {};
\node[label={above:{\small $10$}}] (j) at (2.5,1.5) {};
\node[label={above:{\small $11$}}] (k) at (4,1.5) {};
\draw (a) -- (b) -- (c) -- (e) -- (f) -- (g) -- (h) -- (e) -- (i) -- (j) -- (k)
(g) -- (e) -- (j) -- (d) -- (i)
(c) -- (d)
(b) -- (j) ;
\end{tikzpicture}
\caption{An accessible graph $G$} \label{F.accessibleGraph}
\end{subfigure}
\begin{subfigure}[c]{0.45\textwidth}
\centering
\begin{tikzpicture}
\node[label={above:$1$}] (a) at (0,1.5) {};
\node[label={above:$3$}] (b) at (1,1.5) {};
\node[label={above:$5$}] (c) at (2,1.5) {};
\node[label={above:$7$}] (d) at (3,1.5) {};
\node[label={below:$2$}] (e) at (0.5,0) {};
\node[label={below:$4$}] (f) at (1.5,0) {};
\node[label={below:$6$}] (g) at (2.5,0) {};
\draw (0,1.5) -- (0.5,0) -- (1,1.5) -- (1.5,0) -- (2,1.5) -- (2.5,0) -- (3,1.5)
(0.5,0) -- (2,1.5)
(1,1.5) -- (2.5,0);
\end{tikzpicture}
\caption{A non-accessible graph $H$ with $J_H$ unmixed} \label{F.notAccessible}
\end{subfigure}
\caption{}
\end{figure}

On the other hand, the graph $H$ in Figure \ref{F.notAccessible} is not accessible, even if $J_H$ is unmixed. In fact, the cut sets of $H$ are $\mc C(H) = \{\emptyset, \{2\}, \{6\}, \{2,6\}, \{3,5\}, \{2,4,6\}\}$. In particular, $\{3,5\} \in \mc C(H)$, but neither $3$ nor $5$ is a cut vertex of $H$. More in general, the graphs of \cite[Example 2.2]{BMS18}, which include the graph $H$, are not accessible, even if their binomial edge ideal is unmixed.
\end{example}

When needed, we may assume that the graphs are connected, by the following remark.

\begin{remark}\label{R.connected}
Given a graph $G$ with connected components $G_1,\dots,G_c$, the following properties hold.
\begin{itemize}
\item[(i)] $\mc C(G) = \{S_1 \cup \cdots \cup S_c : S_i \in \mc C(G_i)\}$ by definition. Hence, if $S$ is a cut set of $G$, then $S \cap V(G_i) \in \mc C(G_i)$ for every $i$.
\item[(ii)] $J_G$ is unmixed if and only if $J_{G_i}$ is unmixed for every $i=1, \dots, c$. This follows by (i).
\item[(iii)] $G$ is accessible if and only if $G_i$ is accessible for every $i=1,\dots,c$. It follows by (i) and (ii).
\item[(iv)] $J_G$ is Cohen-Macaulay if and only $J_{G_i}$ is Cohen-Macaulay for every $i=1,\dots,c$. In fact, $R/J_G \cong R_1/J_{G_1} \otimes \cdots \otimes R_c/J_{G_c}$, where $R_i = K[x_j,y_j : j \in V(G_i)]$.
\end{itemize}
\end{remark}

\section{A necessary condition for the Cohen-Macaulayness of \texorpdfstring{$J_G$}{J\_G}}

In this section, we are going to prove that if $J_G$ is Cohen-Macaulay, then $G$ is accessible, by using a certain poset associated with $J_G$, introduced by Àlvarez Montaner in \cite[Definition 3.3]{A19}. We recall here its construction.

Let $I$ be a radical ideal of a commutative Noetherian ring containing a field $K$. We define $\mc P_I$ to be the set of all possible sums of ideals in the minimal primary decomposition of $I$, i.e., \[
\mc P_I = \{I_{i_1}+\cdots+I_{i_s} : I_{i_j} \text{ primary component of } I, s>0\}.
\]
We note that every ideal in $\mc P_I$ contains $I$.

\begin{definition}
Let $G$ be a graph on the vertex set $[n]$ and $J_G$ be the binomial edge ideal of $G$ in the polynomial ring $R=K[x_i,y_i : i \in [n]]$.
We define the poset $\mc Q_{J_G}$ associated with $J_G$ whose elements are given by the following procedure:
\begin{itemize}
\item[1.] set $I := J_G$;
\item[2.] add to $\mc Q_{J_G}$ the prime ideals in $\mc P_I$;
\item[3.] for every non-prime ideal $J \in \mc P_I$, set $I := J$ and return to Step 2.
\end{itemize}

We order the elements of $\mc Q_{J_G}$ by reverse inclusion and then add a top element $1_{\mc Q_{J_G}}$ to $\mc Q_{J_G}$, greater than all the other elements.
\end{definition}

We note that, $\mc Q_{J_G}$ is finite because $R$ is Noetherian and every ideal of $\mathcal{P}_I$ contains $I$.

\begin{example}
Let $G$ be the graph in Figure \ref{F.GraphPoset} and denote by $f_{ij} = x_iy_j-x_jy_i$ the generators of $J_G$. Hence, $J_G = (f_{12}, f_{23}, f_{24}, f_{34}, f_{45})$ and its primary decomposition is $J_G = P_0 \cap P_1 \cap P_2 \cap P_3$, where
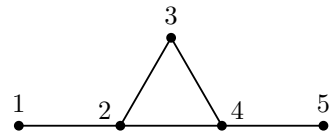
\begin{figure}[ht!]
\begin{minipage}[c]{0.57\textwidth}
    \centering
    \begin{alignat*}{2}
      P_0 &= P_{\emptyset}(G) &=&\, (f_{12}, f_{13}, f_{14}, f_{15}, f_{23}, f_{24}, f_{25}, f_{34}, f_{35}, f_{45}), \\
      P_1 &= P_{\{4\}}(G)     &=&\, (x_4, y_4, f_{12}, f_{13}, f_{23}), \\
      P_2 &= P_{\{2\}}(G)     &=&\, (x_2, y_2, f_{34}, f_{35}, f_{45}), \\
      P_3 &= P_{\{2,4\}}(G)   &=&\, (x_2, y_2, x_4, y_4).
    \end{alignat*}
\end{minipage}%
\begin{minipage}[c]{0.38\textwidth}
    \centering
    \begin{tikzpicture}[scale=0.9]
    \node[label={above:{\small $1$}}] (a) at (0,0) {};
    \node[label={above left:{\small $2$}}] (b) at (1.5,0) {};
    \node[label={above:{\small $3$}}] (c) at (2.25,1.3) {};
    \node[label={above right:{\small $4$}}] (d) at (3,0) {};
    \node[label={above:{\small $5$}}] (e) at (4.5,0) {};
    \draw (0,0) -- (1.5,0) -- (3,0) -- (4.5,0)
    (1.5,0) -- (2.25,1.3) -- (3,0);
    \end{tikzpicture}
    \caption{The graph G} \label{F.GraphPoset}
\end{minipage}
\end{figure}

\noindent Among the sums of the minimal primes of $J_G$, the only non-prime ideal is
\[
P_1 + P_2 = P_1 + P_2 + P_3 = (x_2, x_4, y_2, y_4, f_{13}, f_{35}),
\]
and its primary decomposition is $P_1 + P_2 = Q_0 \cap Q_1$, where
\[
Q_0 = (x_2, x_4, y_2, y_4, f_{13}, f_{15}, f_{35}) \quad \text{ and } \quad Q_1 = (x_2, x_3, x_4, y_2, y_3, y_4).
\]
Moreover, $Q_0 = P_0 + P_3 = P_0 + P_1 + P_2 = P_0 + P_1 + P_3 = P_0 + P_2 + P_3 = P_0 + P_1 + P_2 + P_3$ and $Q_0 + Q_1 = (x_2, x_3, x_4, y_2, y_3, y_4, f_{15})$ are prime ideals. The poset $\mc Q_{J_G}$ is depicted in Figure \ref{F.PosetQ_J_G}.

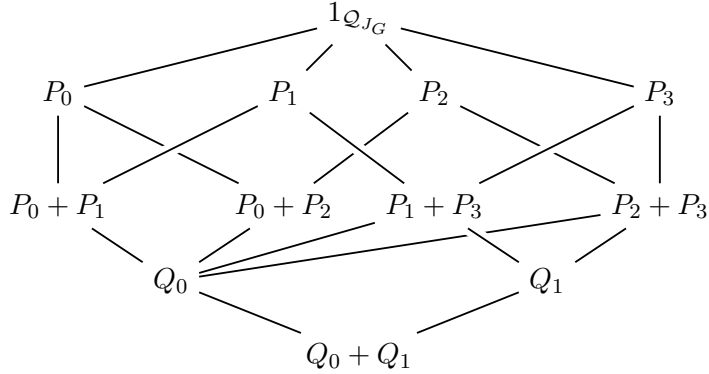
\begin{figure}[ht!]
\tikzstyle{every node}=[font=\normalsize]
\begin{tikzpicture}
  \node (max) at (4,4.5) {$1_{\mc Q_{J_G}}$};
  \node (j) at (8,3.5) {$P_3$};
  \node (i) at (5,3.5) {$P_2$};
  \node (h) at (3,3.5) {$P_1$};
  \node (g) at (0,3.5) {$P_0$};
  \node (f) at (8,2) {$P_2+P_3$};
  \node (e) at (5,2) {$P_1+P_3$};
  \node (d) at (3,2) {$P_0+P_2$};
  \node (c) at (0,2) {$P_0+P_1$};
  \node (b) at (6.5,1) {$Q_1$};
  \node (a) at (1.5,1) {$Q_0$};
  \node (min) at (4,0) {$Q_0+Q_1$};
  \draw (min) -- (a) -- (c) -- (g) -- (max) -- (j) -- (f) -- (b) -- (min)
  (a) -- (d) -- (i) -- (max) -- (h) -- (e) -- (a) -- (f) -- (i)
  (b) -- (e) -- (j)
  (c) -- (h)
  (d) -- (g);
  \draw[preaction={draw=white, -,line width=5pt}] (c) -- (h) -- (e);
  \draw[preaction={draw=white, -,line width=5pt}] (b) -- (e) -- (j);
\end{tikzpicture}
\caption{The poset $\mc Q_{J_G}$} \label{F.PosetQ_J_G}
\end{figure}
\end{example}

\begin{remark} \label{R.Poset}
By construction, every element $I \neq 1_{\mc Q_{J_G}}$ of the poset $\mc Q_{J_G}$ contains at least a prime ideal $P_S(G)$ for some $S \in \mc C(G)$ because $I \in \mc P_J$ for some ideal $J \in \mc Q_{J_G}$ and every ideal of $\mc P_J$ contains $J$. Moreover, if $P_S(G) \subsetneq I$, then $I$ contains another prime ideal $P_U(G)$ with $U \in \mc C(G)\setminus \{S\}$.
\end{remark}

Notice that for every $I_q \in \mc Q_{J_G}$ we have
\[
I_q = P_S(H) = (x_i,y_i : i \in S) + J_{\widetilde H_1} + \cdots + J_{\widetilde H_{c_q}},
\]
for some graph $H$ on the vertex set $[n]$, where $S \subseteq [n]$, $H_1,\dots,H_{c_q}$ are the connected components of $H \setminus S$ and $\widetilde H_i$ is the complete graph on the vertices of $H_i$.
In particular, the poset $\mc Q_{J_G}$ is well-defined because all its elements are radical ideals.
We set
\[
d_q = \dim(R/I_q) = 2n-2|S| - \sum_{i=1}^{c_q} (|H_i|-1) = n-|S|+c_q.
\]
Recall that if $I_q \neq 1_{\mc Q_{J_G}}$, then an open interval of the form $(I_p,I_q) \subsetneq \mc Q_{J_G}$ is the set $\{I_r \in \mc Q_{J_G} : I_q \subsetneq I_r \subsetneq I_p\}$ whereas an open interval of the form $(I_p,1_{\mc Q_{J_G}})$ is the set $\{I_r \in \mc Q_{J_G} : I_r \subsetneq I_p\}$.

In \cite{A19}, \'Alvarez Montaner proves the following topological characterization of Cohen-Macaulay binomial edge ideals, which resembles Reisner's criterion for Cohen-Macaulay squarefree monomial ideals \cite{R76}.

\begin{theorem}[{\cite[Corollary 3.1]{A19}}] \label{T.Poset}
Let $G$ be a graph. The following conditions are equivalent:
\begin{itemize}
\item[{\rm (i)}] $J_G$ is Cohen-Macaulay;
\item[{\rm (ii)}] $\dim_K \widetilde H^{r-d_q-1}((I_q,1_{\mc Q_{J_G}}); K)=0$ for all $r \neq \dim(R/J_G)$ and all $I_q \in \mc Q_{J_G}$.
\end{itemize}
\end{theorem}

Here, $\widetilde H^i((I_q,1_{\mc Q_{J_G}}); K)$ denotes the $i$-th \textit{reduced cohomology group} of the interval $(I_q,1_{\mc Q_{J_G}})$ over the field $K$ (for more details, see \cite[Section 1.5]{W07}). We are now ready to prove the main result of this section.

\begin{theorem} \label{T.CMimpliesRCSP}
Let $G$ be a graph. If $J_G$ is Cohen-Macaulay, then $G$ is accessible.
\end{theorem}

\begin{proof}
In this proof we only deal with cut sets of the graph $G$, hence we simply write $P_S$ in place of $P_S(G)$ for every $S \in \mc C(G)$.

By Remark \ref{R.connected}, we may assume $G$ connected. By contradiction, suppose that $G$ is not accessible. Since $J_G$ is unmixed, there exists a non-empty cut set $S$ of $G$ such that $S \setminus \{s\} \notin \mc C(G)$ for every $s \in S$.
Consider the finite set
\[
\mc A = \{P_S+P_U : U \in \mc C(G), U \subsetneq S \},
\]
which is not empty because $P_S+P_{\emptyset} \in \mc A$. Let $P_S+P_T$ be a minimal element of $\mc A$ with respect to the inclusion. Notice that
\[
P_S + P_T = (x_i,y_i : i \in S) + J_{\widetilde H_1} + \cdots + J_{\widetilde H_c}.
\]
where $H_1,\dots,H_c$ are the connected components of $G \setminus T$ from which we remove the elements of $S \setminus T$ and $c=|T|+1$ since $J_G$ is unmixed. In particular, $P_S + P_T$ is a prime ideal, and hence, it is an element of the poset $\mc Q_{J_G}$.

\textbf{Claim:} There are no ideals $I$ in $\mc Q_{J_G}$ such that $P_S \subsetneq I \subsetneq P_S+P_T$, i.e., the open interval $(P_S+P_T, P_S)$ is empty.

Suppose by contradiction that the claim is not true. Since $P_S \subsetneq I$, by Remark \ref{R.Poset} such an ideal $I$ has to contain at least another prime ideal of the form $P_U$ for some $U \in \mc C(G)$ and $U \neq S$. We distinguish between two cases.
\begin{itemize}
\item[1)] If $U \nsubseteq S$, then there exists $u \in U \setminus S$ and $x_u, y_u \in P_U \subseteq I \subsetneq P_S+P_T$. On the other hand, $x_u, y_u \notin P_S+P_T$ yields a contradiction.
\item[2)] If $U \subsetneq S$, then $P_S + P_U \in \mc A$. It follows that $P_S + P_U \subseteq I \subsetneq P_S + P_T$ which contradicts the minimality of $P_S+P_T$ in $\mc A$.
\end{itemize}

Thus, the claim holds and hence, the point $P_S$ is isolated in the open interval $(P_S+P_T, 1_{\mc Q_{J_G}})$. It follows that the open interval $(P_S+P_T, 1_{\mc Q_{J_G}})$ consists of at least two connected components, one of which is the isolated point $P_S$ and another component containing $P_T$. Hence,
\[
\dim_K \widetilde H^0((P_S+P_T,1_{\mc Q_{J_G}}); K) > 0
\]
because $\dim_K \widetilde H^0(\mc P; K)$ equals the number of connected components of the poset $\mc P$ minus 1 (see \cite[Proposition 2.7 and page 110]{H02}).

Now, set $d = \dim(R/(P_S+P_T)) = n-|S|+c = n-|S|+|T|+1$, where $n=|V(G)|$. Since $G$ is not accessible, we have $|T| < |S| - 1$; thus, $d+1 = n-|S|+|T|+2 < n-1+2 = n+1 = \dim(R/J_G)$, where the last equality follows from the unmixedness of $J_G$. If $r=d+1$, this implies that
\[
\dim_K \widetilde H^{r-d-1}((P_S+P_T,1_{\mc Q_{J_G}}); K) = \dim_K \widetilde H^0((P_S+P_T,1_{\mc Q_{J_G}}); K) > 0.
\]

By Theorem \ref{T.Poset}, it follows that $J_G$ is not Cohen-Macaulay, a contradiction.
\end{proof}

Theorem \ref{T.CMimpliesRCSP} has many consequences on the combinatorics of the graphs with Cohen-Macaulay binomial edge ideal, which we will explore in Section 4. Here we want to show how Theorem \ref{T.CMimpliesRCSP} is useful to prove that a binomial edge ideal is not Cohen-Macaulay.

\begin{example}
In \cite[Examples 2 and 3]{R13} and \cite[Figure 8]{R19}, Rinaldo considers the graph $G$ in Figure \ref{F.UnmixedNonCMGraphs} and the graph $H$ in Figure \ref{F.notAccessible}, showing by symbolic computation that their binomial edge ideals are unmixed and not Cohen-Macaulay. This last fact can be easily shown by Theorem \ref{T.CMimpliesRCSP} just by looking at the cut sets of the two graphs:
\[
\begin{array}{cl}
\mc C(G) =&\!\!\! \{\emptyset,\{2\},\{6\},\{7\},\{2,6\},\{2,7\},\{3,5\},\{3,7\},\{5,6\},\{6,7\}, \{2,3,7\}, \{2,4,6\}, \{2,4,7\}, \{2,5,6\},\\[-2.5mm]
 & \!\! \{2,6,7\}, \{3,5,6\}, \{3,5,7\}, \{2,4,6,7\}\}, \\[-2.5mm]
\mc C(H) =&\!\!\! \{\emptyset,\{2\},\{6\},\{2,6\},\{3,5\},\{2,4,6\}\}.
\end{array}
\]
In both graphs, $\{3,5\}$ is a cut set, but $\{3\}$ and $\{5\}$ are not cut vertices and, hence, $G$ and $H$ are not accessible. Thus, $J_G$ and $J_H$ are not Cohen-Macaulay. Notice that, since $H$ is bipartite, the non-Cohen-Macaulayness of $J_H$ is also a consequence of \cite[Example 5.4]{BMS18}.

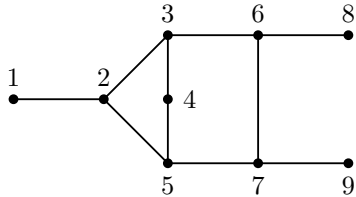
\begin{figure}[ht!]
\begin{tikzpicture}[scale=1.2]
\node[label={above:{\small $1$}}] (a) at (0,0) {};
\node[label={above:{\small $2$}}] (b) at (1,0) {};
\node[label={above:{\small $3$}}] (c) at (1.71,0.71) {};
\node[label={right:{\small $4$}}] (d) at (1.71,0) {};
\node[label={below:{\small $5$}}] (e) at (1.71,-0.71) {};
\node[label={above:{\small $6$}}] (f) at (2.71,0.71) {};
\node[label={below:{\small $7$}}] (g) at (2.71,-0.71) {};
\node[label={above:{\small $8$}}] (h) at (3.71,0.71) {};
\node[label={below:{\small $9$}}] (i) at (3.71,-0.71) {};
\draw (0,0) -- (1,0) -- (1.71,0.71) -- (2.71,0.71) -- (3.71,0.71)
(1,0) -- (1.71,-0.71) -- (1.71,0) -- (1.71,0.71)
(1.71,-0.71) -- (2.71,-0.71) -- (3.71,-0.71)
(2.71,0.71) -- (2.71,-0.71);
\end{tikzpicture}
\caption{The graph $G$}\label{F.UnmixedNonCMGraphs}
\end{figure}
\end{example}

Another interesting application is that \cite[Conjecture 1.6]{BV15} of Benedetti and Varbaro on the diameter of the dual graph holds for all binomial edge ideals, extending \cite[Corollary 6.3]{BMS18}. To explain this, we recall the setting.

Given an ideal $I$ in a polynomial ring $R=K[x_1,\dots, x_n]$, with minimal primes $\mathfrak p_1, \dots, \mathfrak p_r$, the \textit{dual graph}, $\mathcal{D}(I)$, of $I$ is the graph with vertex set $\{\mathfrak p_1, \dots, \mathfrak p_r\}$ and edge set
\[
\{\{\mathfrak p_i, \mathfrak p_j\} : \hgt(\mathfrak p_i+\mathfrak p_j)-1=\hgt(\mathfrak p_i)=\hgt(\mathfrak p_j)=\hgt(I)\}.
\]
The notion of dual graph is implicit in the proof of \textit{Hartshorne Connectedness Theorem} \cite{H62}, which implies that $\mathcal{D}(I)$ is connected if $R/I$ satisfies the Serre's condition $(S_2)$ and, in particular, if $I$ is Cohen-Macaulay.

In \cite[Theorem 5.2]{BMS18}, we describe the dual graph of an unmixed binomial edge ideal $J_G$ in terms of the underlying graph $G$. Moreover, if $G$ is bipartite, we prove that $J_G$ Cohen-Macaulay is equivalent to $\mathcal{D}(J_G)$ connected, which in turn is equivalent to $G$ accessible.

Recall that the \textit{diameter}, ${\rm diam}(G)$, of a graph $G$ is the maximal distance between two of its vertices and a homogeneous ideal is called \textit{Hirsch} if ${\rm diam}(\mathcal D(I)) \leq \hgt(I)$. In \cite[Conjecture 1.6]{BV15}, Benedetti and Varbaro conjecture that every Cohen-Macaulay homogeneous ideal generated in degree two is Hirsch. In \cite[Corollary 6.3]{BMS18}, we essentially prove that $J_G$ is Hirsch if $G$ is accessible. Hence, with the same argument, Theorem \ref{T.CMimpliesRCSP} immediately implies the following result:

\begin{corollary}\label{C.Hirsch}
If $J_G$ is Cohen-Macaulay, then it is Hirsch. In other words, \cite[Conjecture 1.6]{BV15} is true for all binomial edge ideals.
\end{corollary}

\section{Accessible graphs}

In this section we focus on the combinatorial properties of accessible graphs. The purpose is twofold. First, by Theorem \ref{T.CMimpliesRCSP} these turn out to be combinatorial properties of the graphs whose binomial edge ideal is Cohen-Macaulay. Second, the results proved in this section will be of crucial importance in Section \ref{S.Chordal}, where we prove the converse of Theorem \ref{T.CMimpliesRCSP} for chordal and traceable graphs.

\subsection{Combinatorial properties of accessible graphs}

In \cite[Proposition 3.10]{BN17}, it is proved that if $G$ is connected and $J_G$ satisfies the Serre's condition $(S_2)$, then either $G$ is complete or it has at least one cut vertex. In particular, this holds for Cohen-Macaulay binomial edge ideals. In the next lemma we prove that a stronger property holds if $G$ is accessible and hence, a fortiori, if $J_G$ is Cohen-Macaulay (by Theorem \ref{T.CMimpliesRCSP}).

\begin{lemma}\label{L.cutSetsContainCutVertex}
Let $G$ be an accessible graph. Then every non-empty cut set of $G$ contains a cut vertex.
\end{lemma}

\begin{proof}
Let $S \in \mc C(G)$, $S \neq \emptyset$. We proceed by induction on the cardinality of $|S|$. If $|S|=1$, the claim follows. Otherwise, since $G$ is accessible, there exists $s \in S$ such that $S \setminus \{s\} \in \mc C(G)$. By induction there exists a cut vertex $v \in S \setminus \{s\}$ and the same holds for $S$.
\end{proof}

\begin{remark}\label{R.NoCutVerticesComplete}
By Lemma \ref{L.cutSetsContainCutVertex} it follows that, if a graph $G$ is accessible and has no cut vertices, then the only cut set of $G$ is the empty set and the connected components of $G$ are complete by Remark \ref{R.NoCutSetsComplete}.
\end{remark}

\begin{example} \label{E.NonRCSPgraphWithCutVertexInEveryCutSet}
The converse of Lemma \ref{L.cutSetsContainCutVertex} is not true in general. For instance, let $G$ be the graph in Figure \ref{F.Lemma_cutSetsContainingCutVertices} whose cut sets are $\mc C(G) = \{\emptyset, \{7\}, \{8\}, \{7, 8\}, \{6, 8, 9\}, \{7, 8, 9\}, \{6, 7, 8, 9\}\}$.
Notice that every non-empty cut set contains either $7$ or $8$, which are the cut vertices of $G$, but removing any vertex from $\{6,8,9\}$ does not produce a cut set. We also notice that in this case $J_G$ is unmixed and the dual graph of $J_G$ is connected, but $G$ is not accessible and, in particular, $J_G$ is not Cohen-Macaulay.

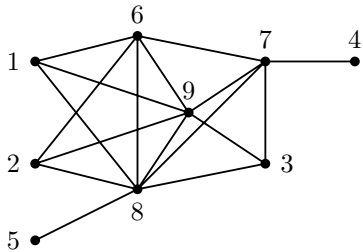
\begin{figure}[ht!]
\centering
\begin{tikzpicture}[scale=1.7]
\node[label={left:{\small $1$}}] (a) at (-0.2,0.4) {};
\node[label={left:{\small $2$}}] (b) at (-0.2,-0.4) {};
\node[label={above:{\small $6$}}] (c) at (0.6,0.6) {};
\node[label={below:{\small $8$}}] (d) at (0.6,-0.6) {};
\node[label={above:{\small $9$}}] (e) at (1,0) {};
\node[label={above:{\small $7$}}] (f) at (1.6,0.4) {};
\node[label={right:{\small $3$}}] (g) at (1.6,-0.4) {};
\node[label={above:{\small $4$}}] (h) at (2.3,0.4) {};
\node[label={left:{\small $5$}}] (i) at (-0.2,-1) {};
\draw (-0.2,-1) -- (0.6,-0.6) -- (-0.2,-0.4) -- (1,0) -- (-0.2,0.4) -- (0.6,0.6) -- (-0.2,-0.4)
(-0.2,0.4) -- (0.6,-0.6) -- (1,0) -- (0.6,0.6) -- (0.6,-0.6) -- (1.6,-0.4) -- (1,0) -- (1.6,0.4) -- (2.3,0.4)
(0.6,0.6) -- (1.6,0.4) -- (0.6,-0.6)
(1.6,-0.4) -- (1.6,0.4);
\end{tikzpicture}
\caption{A graph $G$ with a cut vertex in each non-empty cut set} \label{F.Lemma_cutSetsContainingCutVertices}
\end{figure}
\end{example}

In the next section we will deal with graphs obtained by \textit{completing the neighbourhood of a vertex}. More precisely we recall the following definition:

\begin{definition}
Let $G$ be a graph and let $v$ be a vertex of $G$. We denote by $G_v$ the graph obtained by connecting any two neighbours of $v$, i.e.,
\[
V(G_v)=V(G) \text{ and } E(G_v)=E(G) \cup \{ \{u,w\} : u,w \in N_G(v), u \neq w \}.
\]
\end{definition}

Several properties of $G$ behave well with respect to completing the neighbourhood of a vertex, including being accessible.

\begin{lemma}\label{L.completingNeighborsVertex}
Let $v$ be a vertex of a graph $G$. Then the following properties hold:
\begin{itemize}
\item[\rm (1)] $\mc C(G_v) = \{S \in \mc C(G) : v \notin S\}$;
\item[\rm (2)] if $J_G$ is unmixed, then $J_{G_v}$ is unmixed;
\item[\rm (3)] if $G$ is accessible, then $G_v$ is accessible.
\end{itemize}
\end{lemma}

\begin{proof}
Notice that for every $W \subseteq V(G)$ with $v \notin W$ we have
\[
c_G(W) = c_{G_v}(W).
\]
(1) Let $S \in \mc C(G)$ and $v \notin S$. Then, for every $w \in S$, $c_{G_v}(S \setminus \{w\}) = c_G(S \setminus \{w\}) < c_G(S) = c_{G_v}(S)$. Conversely, let $S \in \mc C(G_v)$. Hence, $v \notin S$ since $v$ is a free vertex of $G_v$ (see \cite[Proposition 2.1]{RR14}). Thus, for every $w \in S$, $c_G(S \setminus \{w\}) = c_{G_v}(S \setminus \{w\}) < c_{G_v}(S) = c_G(S)$.

\noindent (2) and (3) follow from (1) and the formula above. \qedhere
\end{proof}

\begin{example}
The converse of (2) and (3) in Lemma \ref{L.completingNeighborsVertex} does not hold. In fact, let $G$ be the graph in Figure \ref{F.squareWhisker} and $v=3$. Then, the graph $G_v$ has edge set $E(G_v) \! =\! E(G) \cup \{\{2,4\}\}$, see Figure \ref{F.squareWhiskerCompletingNv}. Notice that $\mc C(G) \! =\! \{\emptyset,\! \{2\},\! \{2,4\},\! \{3,5\}\}$ and $\mc C(G_v) \! = \! \{\emptyset,\! \{2\},\! \{2,4\}\}$. In this case, $G_v$ is accessible, but $J_G$ is not unmixed and, hence, $G$ is not accessible. In this case, $J_{G_v}$ is also Cohen-Macaulay.

\begin{center}
\begin{figure}[ht!]
\begin{subfigure}[c]{0.3\textwidth}
\centering
\begin{tikzpicture}[scale=0.9]
\node[label={above:{\small $1$}}] (a) at (0,1.5) {};
\node[label={above:{\small $2$}}] (b) at (1.5,1.5) {};
\node[label={below:{\small $5$}}] (c) at (1.5,0) {};
\node[label={above:{\small $3$}}] (d) at (3,1.5) {};
\node[label={below:{\small $4$}}] (e) at (3,0) {};
\draw (0,1.5) -- (1.5,1.5) -- (3,1.5) -- (3,0) -- (1.5,0) -- (1.5,1.5);
\end{tikzpicture}
\caption{The graph $G$} \label{F.squareWhisker}
\end{subfigure}
\begin{subfigure}[c]{0.3\textwidth}
\centering
\begin{tikzpicture}[scale=0.9]
\node[label={above:{\small $1$}}] (a) at (0,1.5) {};
\node[label={above:{\small $2$}}] (b) at (1.5,1.5) {};
\node[label={below:{\small $5$}}] (c) at (1.5,0) {};
\node[label={above:{\small $3$}}] (d) at (3,1.5) {};
\node[label={below:{\small $4$}}] (e) at (3,0) {};
\draw (0,1.5) -- (1.5,1.5) -- (3,1.5) -- (3,0) -- (1.5,0) -- (1.5,1.5) -- (3,0);
\end{tikzpicture}
\caption{The graph $G_3$} \label{F.squareWhiskerCompletingNv}
\end{subfigure}
\caption{} \label{F.Lemma_completingNeighborsV}
\end{figure}
\end{center}
\end{example}

Another useful operation is the cone from a new vertex over a graph.

\begin{definition}
Let $G$ be a graph and let $v \notin V(G)$. The {\it cone} of $v$ on $G$, denoted by $\mathrm{cone}(v,G)$ is the graph with vertex set $V(G) \cup \{v\}$ and edge set $E(G) \cup \{\{u,v\} : u \in V(G)\}$.
\end{definition}

Some properties of binomial edge ideals of cones are studied in \cite{RR14}. Here, we show that the cone over two disjoint graphs preserves unmixedness, accessibility and Cohen-Macaulayness, proving, in particular, the converse of \cite[Theorem 3.8]{RR14}.

\begin{theorem}\label{T.CMcones}
Let $H_1$ and $H_2$ be connected graphs, $H=H_1 \sqcup H_2$ and let $G = \cone(v,H)$. Then
\begin{itemize}
\item[{\rm (1)}] $\mc C(G) = \{\emptyset\} \cup \{T_1 \sqcup T_2 \sqcup \{v\} : T_i \in \mc C(H_i)\}$;
\item[{\rm (2)}] $J_{H_1}$ and $J_{H_2}$ are unmixed if and only if $J_G$ is unmixed;
\item[{\rm (3)}] $H_1$ and $H_2$ are accessible if and only if $G$ is accessible;
\item[{\rm (4)}] $J_{H_1}$ and $J_{H_2}$ are Cohen-Macaulay if and only if $J_G$ is Cohen–Macaulay.
\end{itemize}
\end{theorem}

\begin{proof}
(1) and (2) are proved in \cite[Lemma 3.5]{RR14} and \cite[Corollary 3.7]{RR14} respectively.
\begin{enumerate}
\item[{\rm (3)}] We first assume that $H_1,H_2$ are accessible. Let $T_1 \sqcup T_2 \sqcup \{v\} \in \mc C(G)$, with $T_i \in \mc C(H_i)$ for $i=1,2$ and, without loss of generality, assume that $T_1 \neq \emptyset$. By assumption, there exists $w \in T_1$ such that $T_1 \setminus \{w\} \in \mc C(H_1)$, thus $(T_1 \setminus \{w\}) \sqcup T_2 \sqcup \{v\} \in \mc C(G)$.\\
      Conversely, suppose that $G$ is accessible. Let $T_1\in \mc C(H_1)\setminus \{\emptyset\}$, then $T=T_1 \sqcup \{v\} \in \mc C(G)$. By assumption, there exists $w \in T$ such that $T \setminus \{w\} \in \mc C(G)$. Since $|T| \geq 2$ and by (1), it follows that $w \neq v$, hence $w \in T_1$. We conclude that $T_1 \setminus \{w\} \in \mc C(H_1)$.
\item[{\rm (4)}] If $J_{H_1}$ and $J_{H_2}$ are Cohen-Macaulay, then $J_G$ is Cohen-Macaulay by \cite[Theorem 3.8]{RR14}. Conversely, assume $J_G$ Cohen-Macaulay. Set $|V(H)|=n$ and $R, R_H, R_{H_i}$ be the polynomial rings corresponding respectively to $G, H, H_i$, for $i=1,2$. Then $\dim(R/J_G)=\depth(R/J_G)=|V(G)|+1=n+2$.\\
      By \cite[Lemma 3.6]{RR14}, $\dim(R/J_G)=\max \{\dim(R_{H_1}/J_{H_1})+\dim(R_{H_2}/J_{H_2}), n+2\}$. Thus,
      \[
      \dim(R_H/J_H)=\dim(R_{H_1}/J_{H_1})+\dim(R_{H_2}/J_{H_2}) \leq n+2.
      \]
      Moreover, by \cite[Theorem 3.9]{KS19}, we have $\depth(R/J_G)=\min \{\depth(R_{H}/J_{H}), n+2\}$. Thus,
      \[
      \depth(R_H/J_H) \geq n+2.
      \]
      We conclude that $\depth(R_H/J_H) = \dim(R_H/J_H) = n+2$, hence $J_H$ is Cohen-Macaulay. By Remark \ref{R.connected} (iv), this is equivalent to have $J_{H_1}$ and $J_{H_2}$ Cohen-Macaulay. \qedhere
\end{enumerate}
\end{proof}

We are interested in the cone operation because an accessible graph containing exactly a cut vertex is a cone.

\begin{lemma}\label{L.oneCutVertex}
Let $G$ be a connected graph with exactly one cut vertex $v$ and let $H_1$ and $H_2$ be the connected components of $G \setminus \{v\}$. If $G$ is accessible, then $G=\mathrm{cone}(v,H_1 \sqcup H_2)$ and $H_1$, $H_2$ are accessible. Moreover, if $J_G$ is Cohen-Macaulay, then also $J_{H_1}$ and $J_{H_2}$ are Cohen-Macaulay.
\end{lemma}

\begin{proof}
By Lemma \ref{L.cutSetsContainCutVertex} all non-empty cut sets of $G$ contain $v$. Then, Proposition \ref{L.completingNeighborsVertex} (1) implies that $\mc C(G_v)=\{\emptyset\}$ and this means that $G_v$ is a complete graph by Remark \ref{R.NoCutSetsComplete}. Hence, $G=\mathrm{cone}(v,H_1 \sqcup H_2)$. It is now enough to apply Theorem \ref{T.CMcones}.
\end{proof}

We now explore some structural properties of accessible graphs.

\begin{proposition}\label{P.subgraphCutVerticesConnected}
Let $G$ be a connected graph with $k$ cut vertices, $v_1,\dots,v_k$. If $G$ is accessible, then the induced subgraph $G[{\{v_1,\dots,v_k\}}]$ is connected.
\end{proposition}

\begin{proof}
We proceed by induction on $k \geq 1$. If $k=1$, the claim is trivial. Let $k>1$, set $H_0=G_{v_k}$, and $H_i=(H_{i-1})_{v_i}$ for $i=1,\dots,k-1$. We notice that for every $i=0,\dots,k-1$, $H_i$ has exactly $k-1-i$ cut vertices, $v_{i+1},\dots,v_{k-1}$, and is accessible by Lemma \ref{L.completingNeighborsVertex} (3). By induction, the induced subgraph $H_0[\{v_1,\dots,v_{k-1}\}]$ is connected and it is enough to show that $v_k$ is adjacent to some $v_i$ in $H_0$. Since $v_{k-1}$ is the only cut vertex of $H_{k-2}$, by Lemma \ref{L.oneCutVertex}, $v_{k-1}$ is adjacent to $v_k$ in $H_{k-2}$. Hence, either $v_k$ is adjacent to $v_{k-1}$ in $H_0$ or it is adjacent to some other cut vertex $v_i$, with $1 \leq i <k-1$, in $H_0$.
\end{proof}

\begin{proposition} \label{P.EveryVertexIsAdjacentToACutVertex}
Let $G$ be a non-complete connected graph and suppose that every non-empty cut set of $G$ contains a cut vertex. Then, every vertex of $G$ that is not a cut vertex is adjacent to a cut vertex.
\end{proposition}

\begin{proof}
Since $G$ is not complete, by Remark \ref{R.NoCutSetsComplete} it has some non-empty cut sets and, hence, it has at least one cut vertex by assumption.
Assume that there exists a vertex $w$ of $G$ which is not a cut vertex and is not adjacent to any cut vertex. Let $v_1, \dots, v_r$ be the cut vertices of $G$ for some $r \geq 1$. Define $N=N_G[v_1] \cup N_G[v_2] \cup \dots \cup N_G[v_r]$ and $N'=\{v \in N : N_G(v) \nsubseteq N\} \subseteq (N \setminus \{v_1, \dots, v_r\})$. Therefore, every $v \in N'$ is adjacent to some $x \notin N$ and to some $y \in \{v_1, \dots, v_r\}$, and by construction $x$ and $y$ belong to two different connected components of $G \setminus N'$. Thus, $N'$ is a cut set of $G$. Moreover, $N'$ is not empty since otherwise $N=V(G)$ against $w \in V(G) \setminus N$. By construction, $N'$ does not contain any cut vertex, and this contradicts the assumption.
\end{proof}

Let $G$ be a connected and accessible graph. If $G$ has no cut vertices, it is a complete graph by Remark \ref{R.NoCutVerticesComplete}. If $G$ has one cut vertex, it is a cone over an accessible graph with fewer vertices by Lemma \ref{L.oneCutVertex}. The next statement summarizes some properties when $G$ has at least two cut vertices.

\begin{theorem}\label{T.propertiesCM}
Let $G$ be a connected accessible graph with at least two cut vertices. Then
\begin{itemize}
\item[$(1)$] every non-empty cut set of $G$ contains a cut vertex;
\item[$(2)$] the graph induced on the cut vertices of $G$ is connected;
\item[$(3)$] every vertex of $G$ is adjacent to a cut vertex.
\end{itemize}
In particular, these properties hold if $J_G$ is Cohen-Macaulay.
\end{theorem}

\begin{proof}
$(1)$ is Lemma \ref{L.cutSetsContainCutVertex}, $(2)$ is Proposition \ref{P.subgraphCutVerticesConnected}, whereas $(3)$ follows by $(2)$ and Proposition \ref{P.EveryVertexIsAdjacentToACutVertex}. The last part of the claim follows by Theorem \ref{T.CMimpliesRCSP}.
\end{proof}

\begin{remark}
Notice that properties (1), (2) and (3) in Theorem \ref{T.propertiesCM} are only necessary but not sufficient for $G$ to be accessible. In fact, the graph $G$ in Example \ref{E.NonRCSPgraphWithCutVertexInEveryCutSet} satisfies the above three properties but $G$ is not accessible.
\end{remark}

\subsection{Cut sets of accessible graphs}

We now study the structure of the cut sets of accessible graphs. First we provide new combinatorial interpretations of accessibility. We start with a preliminary result.

\begin{lemma}\label{L.PropertiesVerticesOfCutSets}
Let $G$ be a graph with $J_G$ unmixed, $S$ be a cut set of $G$ and $s \in S$. Then, $S \setminus \{s\}$ is a cut set of $G$ if and only if $s$ reconnects exactly two connected components of $G \setminus S$.
\end{lemma}

\begin{proof}
Let $c$ be the number of connected components of $G$ and assume that $S\setminus \{s\} \in \mathcal{C}(G)$. Since $J_G$ is unmixed, $c_G(S)=|S|+c$ and $c_G(S \setminus \{s\})=|S|-1+c$; hence, $s$ reconnects exactly two connected components of $G \setminus S$.

Conversely, let $G_1, \dots, G_{r+c}$ be the connected components of $G \setminus S$ and assume that $s$ reconnects only $G_1$ and $G_2$.
Let us consider the set
\[
Z=\{z \in S : z \text{ is not adjacent to any vertex of $G_3 \cup \dots \cup G_{r+c}$}\},
\]
which contains $s$. Then, $T=S \setminus Z$ is a cut set of $G$ and the connected components of $G \setminus T$ are $G[V(G_1 \cup G_2) \cup Z]$, $G_3, \dots$, $G_{r+c}$. The unmixedness of $J_G$ implies that $|S\setminus Z|=r-1$; then, $Z = \{s\}$ and $S \setminus \{s\}$ is a cut set of $G$.
\end{proof}

\begin{corollary} \label{C.rcsp}
Let $G$ be a graph. The following conditions are equivalent:
\begin{itemize}
\item[{\rm (1)}] $G$ is accessible;
\item[{\rm (2)}] $J_G$ is unmixed and it is possible to order every  $S \in \mc C(G)$ in such a way that $S=\{s_1, \dots, s_r\}$ and $\{s_1, \dots, s_i\} \in \mc C(G)$ for every $i=1, \dots, r$;
\item[{\rm (3)}] It is possible to order every $S \in \mc C(G)$ in such a way that $S=\{s_1, \dots, s_r\}$ and $c_G(\{s_1,\dots,s_{i}\})=c_G(\{s_1, \dots, s_{i-1}\})+1$ for every $i=1, \dots, r$.
\end{itemize}
\end{corollary}

\begin{proof}
The equivalence between (1) and (2) follows by the definition of accessible graph.
Let $S$ be a cut set of $G$ with cardinality $r$ and let $c$ be the number of the connected components of $G$.
We first notice that (3) implies the unmixedness of $J_G$: indeed, $c_G(S)=c_G(\{s_1, \dots, s_{r-1}\})+1=c_G(\{s_1, \dots, s_{r-2}\})+2= \dots =c_G(\emptyset)+r=c+|S|$. Now, the equivalence between (2) and (3) is a consequence of Lemma \ref{L.PropertiesVerticesOfCutSets}.
\end{proof}

Given a cut set $S$ of an accessible graph $G$, by Corollary \ref{C.rcsp} we know that there exists an order of the elements of $S=\{s_1, \dots, s_r\}$ such that $\{s_1, \dots, s_i\}$ is a cut set of $G$ for every $i=1, \dots, r$, but in general we do not have control on how this order can be chosen. The following results of this section allow to fill this gap.

\begin{lemma}\label{L.cutSetsConsistingOfCutVertices}
Let $G$ be a graph with $J_G$ unmixed and let $S$ be a cut set of $G$. If every element of $S$ is a cut vertex of $G$, then every subset of $S$ is a cut set of $G$.
\end{lemma}

\begin{proof}
We may assume that $G$ is connected by Remark \ref{R.connected} and we proceed by induction on $|S|=r$. If $r \leq 2$ the claim holds, hence we fix $r \geq 3$. It is enough to show that all subsets of $S$ with cardinality $r-1$ are cut sets. Thus, assume by contradiction that there exists $s \in S$ such that $S\setminus \{s\}$ is not a cut set of $G$ and let $t \in S \setminus \{s\}$ such that $c_G(S\setminus\{s\})=c_G(S \setminus \{s,t\})$.
Let $G_1, \dots, G_{r+1}$ be the connected components of $G \setminus S$ and assume that $s$ reconnects exactly the components $G_1, \dots, G_p$ for some $p \geq 2$. Since $c_G(S\setminus\{s\})=c_G(S \setminus \{s,t\})$ and $t$ reconnects at least two connected components of $G \setminus S$, the vertex $t$ is not adjacent to vertices of $G_{p+1} \cup \dots \cup G_{r+1}$ in $G \setminus S$.
Consider the set
\[
Z=\{z \in S : z \text{ is not adjacent to vertices of } G_{p+1} \cup \dots \cup G_{r+1}\}.
\]
Clearly $s,t \in Z$ and $T=S \setminus Z \in \mc C(G)$ by construction.

Notice that $G_1, \dots, G_p$ are in the same connected component of $G \setminus \{t\}$ because $s$ is adjacent to vertices of each of $G_1, \dots, G_p$. Since $t$ is a cut vertex of $G$, there exists a vertex $u \in S$, which is adjacent to $t$ and which is not adjacent to any vertex of $G_1 \cup \dots \cup G_p$. Moreover, $u$ reconnects at least two connected components of $G \setminus S$ among $G_{p+1},\dots,G_{r+1}$, say $G_{p+1}$ and $G_{p+2}$. In particular, $u \in T$ and, thus, $T\setminus \{u\} \in \mc C(G)$ by induction. On the other hand, $u$ reconnects at least $G_{p+1}$, $G_{p+2}$ and the connected component containing $t$ in $G \setminus T$ (which is $G[V(G_1 \cup \cdots \cup G_p) \cup Z]$). Thus, Lemma \ref{L.PropertiesVerticesOfCutSets} yields a contradiction.
\end{proof}

\begin{example}
Lemma \ref{L.cutSetsConsistingOfCutVertices} does not hold if we do not require $J_G$ unmixed. In fact, let $G$ be the graph in Figure \ref{F.Lemma_cutSetsConsistingOfCutVertices}. Clearly $J_G$ is not unmixed, since $c_G(\{6\})=3 \neq |\{6\}|+1$. We notice that $2,4,6$ are cut vertices of $G$ and $\{2,4,6\} \in \mc C(G)$, but $\{4,6\} \notin \mc C(G)$ since $c_G(\{4,6\})=c_G(\{6\})$.

\begin{figure}[ht!]
\centering
\begin{tikzpicture}
\node[label={above:{\small $1$}}] (a) at (0,0) {};
\node[label={above:{\small $2$}}] (b) at (1,0) {};
\node[label={below:{\small $5$}}] (c) at (1.71,-0.71) {};
\node[label={above:{\small $3$}}] (d) at (1.71,0.71) {};
\node[label={above:{\small $4$}}] (e) at (2.41,0) {};
\node[label={above:{\small $6$}}] (f) at (3.41,0) {};
\node[label={right:{\small $7$}}] (g) at (4.3,0.6) {};
\node[label={right:{\small $8$}}] (h) at (4.3,-0.6) {};
\draw (0,0) -- (1,0) -- (1.71,0.71) -- (2.41,0) -- (3.41,0) -- (4.3,0.6)
(3.41,0) -- (4.3,-0.6)
(1,0) -- (1.71,-0.71) -- (2.41,0);
\end{tikzpicture}
\caption{The graph $G$} \label{F.Lemma_cutSetsConsistingOfCutVertices}
\end{figure}
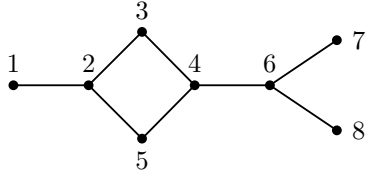
\end{example}

\begin{proposition}\label{P.removingNonCutVertex2}
Let $G$ be a graph and $S = \{v_1,\dots,v_h, w_1,\dots,w_k\} \in \mc C(G)$, where $k \geq 1$, $v_1,\dots,v_h$ are cut vertices of $G$ and $w_1,\dots,w_k$ are not cut vertices of $G$. If $G$ is accessible, then $S \setminus \{w_i\} \in \mc C(G)$ for some $i \in \{1,\dots,k\}$.
\end{proposition}

\begin{proof}
We may assume that $G$ is connected by Remark \ref{R.connected} and $h \geq 1$ by Lemma \ref{L.cutSetsContainCutVertex}. We proceed by induction on the number of cut vertices in $S$, $h \geq 1$. If $h=1$ the claim follows by the definition of accessible graph and by Lemma \ref{L.cutSetsContainCutVertex}. Hence, assume $h>1$. Let $G_1, \dots, G_{h+k+1}$ be the connected components of $G \setminus S$ and suppose by contradiction that $S \setminus \{w_i\} \notin \mc C(G)$ for $i=1, \dots, k$.

Since $G$ is accessible, there exists $v_i \in S$, say $v_1$, such that $S \setminus \{v_1\} \in \mc C(G)$. By induction, we may assume without loss of generality that $S\setminus \{v_1,w_1 \} \in \mc C(G)$.

By Lemma \ref{L.PropertiesVerticesOfCutSets}, $v_1$ reconnects exactly two components of $G \setminus S$, say $G_{1}$ and $G_{2}$. Since $S \setminus \{v_1, w_1\} \in \mc C(G)$, $w_1$ reconnects two components of $G \setminus (S\setminus \{v_1\})$. If it only reconnects $G_{i_1}$ and $G_{i_2}$ with $i_1$, $i_2 \geq 3$, then it reconnects exactly two connected components also in $G\setminus S$ and, then, $S \setminus \{w_1\}$ would be a cut set. Thus, $w_1$ is adjacent to some vertices of $G_1$, of $G_2$, and of another component, say $G_3$. This implies that the connected components of $G \setminus (S\setminus \{v_1,w_1\})$ are $H=G[V(G_1 \cup G_2 \cup G_3) \cup \{v_1,w_1\}]$ and $G_4, \dots, G_{h+k+1}$.

If $k>1$, by induction there exists a vertex $w_j \in S$, say $w_2$, such that $S \setminus \{v_1,w_1, w_2\} \in \mc C(G)$.
As before, $w_2$ is adjacent to some vertex of $G_1 \cup G_2 \cup G_3$ and of another $G_i$, say $G_4$. Iterating this process, we obtain that $T = S\setminus \{v_1,w_1, \dots, w_k\}$ is a cut set of $G$ and, up to relabeling the $G_j$'s, the connected components of $G \setminus T$ are $H'= G[V(G_1 \cup G_2 \cup G_3 \cup G_{4} \cup \dots \cup G_{k+2}) \cup \{v_1, w_1, \dots, w_k\}]$ and $G_{k+3}, \dots, G_{h+k+1}$.

By construction, $H' \setminus \{v_1\}$ is connected.
On the other hand, $v_1$ is a cut vertex of $G$, thus $v_1$ is adjacent to some $v_r$ with $r \in \{2,\dots,h\}$, where $v_r$ is not adjacent to vertices of $H'\setminus \{v_1\}$, otherwise removing $v_1$ from $G$ would not disconnect $G$.
Moreover, since $T$ consists only of cut vertices, Lemma \ref{L.cutSetsConsistingOfCutVertices} implies that also $T \setminus \{v_r\} \in \mc C(G)$ and, thus, $v_r$ reconnects exactly two connected components of $G \setminus T$ by Lemma \ref{L.PropertiesVerticesOfCutSets}.
Therefore, in $G \setminus T$ the vertex $v_r$ reconnects $H'$ to exactly one other connected component $G_q$ with $k+3 \leq q \leq h+k+1$. Consequently, since $v_r$ is not adjacent to any vertex of $H'\setminus \{v_1\}$, if we add back $v_r$ to $G \setminus S$, we have that $v_r$ is adjacent only to some vertex of $G_q$ and, hence, $c_G(S)=c_G(S\setminus \{v_r\})$, which is a contradiction.
\end{proof}

\begin{example}
Let $G$ be the graph in Figure \ref{F.accessibleGraph}. In Example \ref{E.accessible} we saw that $G$ is accessible. Let us consider the cut set $S = \{2,5,10, 4,7\}$, where $2,5,10$ are cut vertices and $4,7$ are not cut vertices of $G$. Since $G$ is accessible, by Proposition \ref{P.removingNonCutVertex2}, we can get a new cut set by removing a non-cut vertex from $S$: indeed $T = S \setminus \{7\} = \{2,5,10, 4\} \in \mc C(G)$ and, applying again the same result, we still get a cut set by removing the only non-cut vertex from $T$, i.e., $U = T \setminus \{4\} = \{2,5,10\} \in \mc C(G)$. Now, $U$ only consists of cut vertices and by Lemma \ref{L.cutSetsConsistingOfCutVertices}, every subset of $U$ is a cut set of $G$.
\end{example}

\section{Strongly unmixed binomial edge ideals}

In Section 3 we saw that the accessibility of $G$ is necessary for the Cohen-Macaulayness of $J_G$. In order to study the remaining implication of Conjecture \ref{C.mainConj}, we present a new combinatorial condition, called strong unmixedness, which turns out to be sufficient for Cohen-Macaulayness. In the next section, we are going to show that for large classes of graphs, being accessible is equivalent to the strong unmixedness of the binomial edge ideal, thus proving the conjecture for those graphs.

\medskip
Let $G$ be a graph and let $v$ be a vertex of $G$. We can decompose $J_G$ as $J_G = A \cap B$, where
\[
A = \bigcap_{\substack{S \in \mc C(G)\\ v \notin S}} P_S(G) \quad \text{ and } \quad B = \bigcap_{\substack{S \in \mc C(G)\\ v \in S}} P_S(G).
\]
Since $J_G = A \cap B$, we have the following short exact sequence:
\begin{equation}\label{Eq.shortExactSequenceFinal}
\tag{$\star$} 0 \longrightarrow R/J_G \longrightarrow R/A \oplus R/B \longrightarrow R/(A+B) \longrightarrow 0.
\end{equation}

By Lemma \ref{L.completingNeighborsVertex} (1) it is clear that $A=J_{G_v}$. The above short exact sequence has been used in other papers about binomial edge ideals, however the structure of $B$ and $A+B$ is not known in general. Our next goal is to describe these two ideals when $v$ is a cut vertex of $G$ and $J_{G \setminus \{v\}}$ is unmixed.

\begin{notation}
Let $v$ be a cut vertex of $G$ and $H_1$ a connected component of $G\setminus \{v\}$. With abuse of notation, we denote by $N_{H_1}(v)$ the set $\{w \in V(H_1) : \{v,w\} \in E(G)\}$.
\end{notation}

\begin{proposition} \label{P.J_G-v_unmixed}
Let $v$ be a cut vertex of a connected graph $G$ and assume that $J_G$ is unmixed. Let $H_1$ and $H_2$ denote the connected components of $H=G\setminus \{v\}$. The following statements are equivalent:
\begin{enumerate}
\item[$\mathrm{(1)}$] $J_{H}$ is unmixed;
\item[$\mathrm{(2)}$] if $S \in \mc C(H)$, then $N_{H_1}(v) \not\subseteq S$ and $N_{H_2}(v) \not\subseteq S$;
\item[$\mathrm{(3)}$] $\mc C(H)=\{S \subseteq V(H) : S \cup \{v\} \in \mc C(G)\}$.
\end{enumerate}

If the above conditions hold, then the ideals in the sequence \eqref{Eq.shortExactSequenceFinal} are $B=(x_v,y_v)+J_{G\setminus \{v\}}$ and $A+B=(x_v,y_v)+J_{G_v\setminus \{v\}}$.
\end{proposition}

\begin{proof}
$(1) \Rightarrow (2)$: Assume by contradiction that there exists $S \in \mc C(H)$ such that $N_{H_1}(v) \subseteq S$. In particular, $S'=S \cap V(H_1)$ is a cut set of $H$ by Remark \ref{R.connected}. Moreover, $S'$ is also a cut set of $G$ because $N_{H_1}(v) \subseteq S'$. Since $J_G$ and $J_{H}$ are unmixed and $H$ has two connected components, we have
\[
|S'|+1=c_{G}(S')=c_{H}(S')=|S'|+2,
\]
which is a contradiction. \\
$(2) \Rightarrow (3)$: It is clear that $\{S \subseteq V(H) : S \cup \{v\} \in \mc C(G)\} \subseteq \mc C(H)$. Conversely, let $S$ be a cut set of $H$. Since $H \setminus S=G \setminus (S \cup \{v\})$ for every $s \in S$, we have
\[
c_G((S \cup \{v\}) \setminus \{s\})=c_H(S \setminus \{s\}) < c_{H}(S)=c_G(S \cup \{v\}).
\]
Moreover, $c_G(S)<c_G(S \cup \{v\})$ because $N_{H_1}(v) \not\subseteq S$ and $N_{H_2}(v) \not\subseteq S$. Thus, $S\cup \{v\}\in \mc C(G)$. \\
$(3) \Rightarrow (1)$: Since $H \setminus S=G \setminus (S \cup \{v\})$ for every $S \in \mc C(H)$, the claim follows by the unmixedness of $J_G$.

The last part of the statement is an easy consequence of the definition of $B$ and of the fact that $A = J_{G_v}$.
\end{proof}

\begin{example}\label{E.(G92_ragno)-vNotUnmixed}
If $G$ is a connected graph with $J_G$ unmixed and $v$ is a cut vertex of $G$, it is not always true that $J_{G \setminus \{v\}}$ is unmixed. For example, given the graph $G$ in Figure \ref{F.G92StronglyUnmixed} one can check that $J_G$, $J_{G \setminus \{2\}}$ and $J_{G \setminus \{8\}}$ are unmixed, but $J_{G \setminus \{7\}}$ is not unmixed. In fact, $\{3,4,8\} \in \mc C(G \setminus \{7\})$ and $c_{G \setminus \{7\}}(\{3,4,8\}) = 4 \neq |\{3,4,8\}| + 2$.

It can also be that no cut vertex works. Let $F$ be the graph in Figure \ref{F.G-vNotUnmixed}, whose cut vertices are $2$ and $6$. By \cite[Example 2.2]{BMS18}, $J_F$ is unmixed, but $J_{F \setminus \{2\}}$ and $J_{F \setminus \{6\}}$ are not unmixed since $c_{F \setminus \{2\}}(\{3,5\}) = c_{F \setminus \{6\}}(\{3,5\}) = 3 \neq |\{3,5\}| + 2$.

Notice that in these cases, when $J_{G \setminus \{v\}}$ is not unmixed, $B$ does not have the form $(x_v,y_v) + J_{G\setminus \{v\}}$ (see Proposition \ref{P.J_G-v_unmixed}).

\vspace{-2mm}
\begin{figure}[ht!]
\begin{subfigure}[c]{0.45\textwidth}
\centering
\begin{tikzpicture}[scale=0.95]
\node[label={below:{\small $1$}}] (a) at (-0.5,0) {};
\node[label={below:{\small $2$}}] (b) at (1,0) {};
\node[label={below:{\small $3$}}] (c) at (2.5,0) {};
\node[label={above:{\small $4$}}] (d) at (0,0.75) {};
\node[label={above:{\small $5$}}] (e) at (1.25,0.75) {};
\node[label={above:{\small $6$}}] (f) at (-0.5,1.5) {};
\node[label={above:{\small $7$}}] (g) at (1,1.5) {};
\node[label={above:{\small $8$}}] (h) at (2.5,1.5) {};
\node[label={above:{\small $9$}}] (h) at (4,1.5) {};
\draw (-0.5,0) -- (1,0) -- (2.5,0) -- (1.25,0.75) -- (0,0.75) -- (1,0) -- (2.5,1.5) -- (1,1.5) -- (0,0.75)
(-0.5,1.5) -- (1,1.5)
(2.5,0) -- (2.5,1.5) -- (4,1.5)
(1.25,0.75) -- (2.5,1.5);
\end{tikzpicture}
\caption{A graph $G$ with $J_{G \setminus \{7\}}$ not unmixed} \label{F.G92StronglyUnmixed}
\end{subfigure}
\begin{subfigure}[c]{0.45\textwidth}
\centering
\begin{tikzpicture}[scale=0.95]
\node[fill=white, draw=white] (h) at (0,2.3) {};
\node[label={above:$1$}] (a) at (0,1.5) {};
\node[label={above:$3$}] (b) at (1,1.5) {};
\node[label={above:$5$}] (c) at (2,1.5) {};
\node[label={above:$7$}] (d) at (3,1.5) {};
\node[label={below:$2$}] (e) at (0.5,0) {};
\node[label={below:$4$}] (f) at (1.5,0) {};
\node[label={below:$6$}] (g) at (2.5,0) {};
\node[fill=white, draw=white] (h) at (0,-0.4) {};
\draw (0,1.5) -- (0.5,0) -- (1,1.5) -- (1.5,0) -- (2,1.5) -- (2.5,0) -- (3,1.5)
(0.5,0) -- (2,1.5)
(1,1.5) -- (2.5,0);
\end{tikzpicture}
\caption{A graph $F$ with $J_{F \setminus \{v\}}$ not unmixed for every cut vertex $v$} \label{F.G-vNotUnmixed}
\end{subfigure}
\vspace{-4mm}
\caption{} \label{F.G-vNotUnmixed2}
\end{figure}
\end{example}

\begin{remark} \label{R.NeighborsCutSet}
Let $v$ be a cut vertex of $G$ and let $H_1$ be a connected component of $G \setminus \{v\}$. If there exists a cut set $S$ of $G\setminus \{v\}$ containing $N_{H_1}(v)$, then $N_{H_1}(v)$ is a cut set of $G$. Indeed, every $w \in N_{H_1}(v)$ is adjacent to $v$ and to some vertices of $V(H_1) \setminus S \subseteq V(H_1)\setminus N_{H_1}(v)$, which is not in the same connected component of $v$ in $G \setminus N_{H_1}(v)$.
\end{remark}

In light of Proposition \ref{P.J_G-v_unmixed}, we now describe the cut sets of $G_v \setminus \{v\}$.

\begin{lemma}\label{L.Gv-vHasRCSP}
Let $G$ be a connected graph with $J_G$ unmixed and let $v$ be a cut vertex such that $J_{G \setminus \{v\}}$ is unmixed. Then
$$\mc C(G_v \setminus \{v\})=\mc C(G_v) \setminus \{S \in \mc C(G_v) : N_G(v) \subseteq S\}.$$
In particular, $J_{G_v \setminus \{v\}}$ is unmixed.
\end{lemma}

\begin{proof}
Let $S \in \mc C(G_v \setminus \{v\})$. It is clear that $S \in \mc C(G_v)$ because $v$ is a free vertex of $G_v$, see \cite[Proposition 2.1]{RR14}.
If $C_1,C_2,\dots,C_r$ are the connected components of $G \setminus S$ and $v \in C_1$, the connected components of $G_v \setminus (\{v\} \cup S)$ are $(C_1)_v \setminus \{v\},C_2,\dots,C_r$, with $(C_1)_v \setminus \{v\}$ possibly empty. Suppose by contradiction that $N_G(v) \subseteq S$. In this case $(C_1)_v \setminus \{v\}=\emptyset$ and the connected components of $G_v \setminus (\{v\} \cup S)$ and $G \setminus (\{v\} \cup  S)$ are $C_2,\dots,C_r$. Clearly $S$ is also a cut set of $G \setminus \{v\}$ and, if $H_1$ and $H_2$ are the connected components of $G \setminus \{v\}$, $S$ is a cut set of $G \setminus \{v\}$ containing $N_{H_1}(v)$ and $N_{H_2}(v)$. This contradicts Proposition \ref{P.J_G-v_unmixed} because $J_{G \setminus \{v\}}$ is unmixed.

Conversely, let $S \in \mc C(G_v)$, i.e., $S \in \mc C(G)$ and $v \notin S$ by Lemma \ref{L.completingNeighborsVertex} (1), and suppose that $N_G(v) \nsubseteq S$. Hence, there exists $w \in N_G(v) \setminus S$ in the connected component $(C_1)_v$ of $G_v \setminus S$. Let $x \in S$ be a vertex adjacent to $v$ that reconnects $(C_1)_v$ to another component $D$ of $G_v \setminus S$. Then, $x$ is adjacent to $w$ in $G_v \setminus \{v\}$ and if we add back $x$ to $(G_v \setminus \{v\}) \setminus S$, it reconnects $(C_1)_v \setminus \{v\}$ to $D$.
This shows that $c_{G_v \setminus \{v\}}(S \setminus \{x\}) < c_{G_v \setminus \{v\}}(S)$ for every $x \in S$, hence $S \in \mc C(G_v \setminus \{v\})$.

As for the last part, let $S \in \mc C(G_v \setminus \{v\})$. We have $(C_1)_v \setminus \{v\} \neq \emptyset$ in $G_v \setminus (\{v\} \cup S)$ because $N_G(v) \not \subseteq S$. Then, $c_{G_v \setminus \{v\}}(S)=c_G(S)=|S|+1$ and $J_{G_v \setminus \{v\}}$ is unmixed.
\end{proof}

We now introduce the notion of strongly unmixed binomial edge ideal, which involves the ideals appearing in the short exact sequence \eqref{Eq.shortExactSequenceFinal}. This sequence will allow us to prove that these ideals are Cohen-Macaulay.

\begin{definition}\label{D.StronglyUnmixed}
Let $G$ be a graph. We say that $J_G$ is {\em strongly unmixed} if the connected components of $G$ are complete graphs or if $J_G$ is unmixed and there exists a cut vertex $v$ of $G$ such that $J_{G \setminus \{v\}}$, $J_{G_v}$ and $J_{G_v \setminus \{v\}}$ are strongly unmixed.
\end{definition}

Strong unmixedness is an inherently combinatorial condition because so is unmixedness.

\begin{remark}
In order to show that $J_G$ is strongly unmixed, we do not need to check the unmixedness of $J_{G_v}$ and $J_{G_v \setminus \{v\}}$ since this follows from the unmixedness of $J_G$ and $J_{G \setminus \{v\}}$ by Lemmas \ref{L.completingNeighborsVertex} and \ref{L.Gv-vHasRCSP}.
\end{remark}

\begin{examples}\label{E.stronglyUnmixed}
\textbf{(a)} Let $n \geq 2$ and $k \geq 1$. Consider the graph $G_{n,k}$ on $n+k-1$ vertices, with edge set
\[
E(G_{n,k})=\{\{1,2\},\{2,3\},\dots,\{n-1,n\}\} \cup \{\{i,j\}: n \leq i<j \leq n+k-1\}.
\]
Notice that $G_{n,k}$ is a complete graph on $k$ vertices with a path on $n$ vertices attached to one of its vertices, see Figure \ref{F.G4,5StronglyUnmixed}; in particular, the graph $G_{n,1}$ is a path on $n$ vertices. We prove that $J_{G_{n,k}}$ is strongly unmixed by induction on $n \geq 2$.

By \cite[Theorem 2.7]{RR14}, $J_{G_{n,k}}$ is unmixed. Let $n=2$. If $k=1$, the graph $G_{2,1}$ is a single edge and we have nothing to show. Let $k>1$; then $G_{2,k}$ is a single edge attached to a complete graph on $k$ vertices. Consider the cut vertex $v=2$. The connected components of $G_{2,k} \setminus \{v\}$ are $\{1\}$ and a complete graph on $k-1$ vertices. The graphs $(G_{2,k})_v$ and $(G_{2,k})_v \setminus \{v\}$ are complete, thus $J_{G_{2,k} \setminus \{v\}}, J_{(G_{2,k})_v}$, and $J_{(G_{2,k})_v \setminus \{v\}}$ are strongly unmixed.

Now, fix $n>2$ and consider the cut vertex $v=n$ of $G_{n,k}$. The connected components of $G_{n,k} \setminus \{v\}$ are a complete graph on $k-1$ vertices and a path on $n-1$ vertices, which is $G_{n-1,1}$. The graph $(G_{n,k})_v$ equals $G_{n-1,k+1}$ and the graph $(G_{n,k})_v \setminus \{v\}$ is isomorphic to $G_{n-1,k}$. For all three graphs, we conclude that the corresponding binomial edge ideal is strongly unmixed by induction on $n$.

In particular, it follows that the binomial edge ideal of a path is strongly unmixed.
\medskip

\textbf{(b)} Let $H$ be the graph in Figure \ref{F.G44StronglyUnmixed}. We show that $J_H$ is strongly unmixed. First of all, notice that $J_H$ is unmixed. In fact, the cut sets of $H$ are
\[
\mc C(H)=\{\emptyset,\{2\},\{6\},\{2,6\},\{2,4,6\}\}
\]
and it is easy to see that $c_H(S)=|S|+1$ for every $S \in \mc C(H)$. We consider the cut vertex $v=6$ and show that $J_{H \setminus \{6\}}$ is strongly unmixed. The connected components of $H \setminus \{6\}$ are the isolated point $7$ and the graph $C = \mathrm{cone}(2, P \sqcup \{1\})$, where $E(P)=\{\{3,4\},\{4,5\}\}$. It suffices to prove that $J_C$ is strongly unmixed. By part \textbf{(a)}, $J_P$ is strongly unmixed, hence in particular unmixed. Thus, by Theorem \ref{T.CMcones} (2), $J_C$ is unmixed. The vertex $2$ is now a cut vertex of $C$. The connected components of $C \setminus \{2\}$ are $P$ and $\{1\}$, hence $J_{C \setminus \{2\}}$ is strongly unmixed by part \textbf{(a)}. Clearly, $J_{C_2}$ and $J_{C_2 \setminus \{2\}}$ are strongly unmixed, since $C_2$ and $C_2 \setminus \{2\}$ are complete graphs. Now, the graph $H_6$ is a complete graph on the vertices $\{2,3,4,5,6,7\}$ with the edge $\{1,2\}$ attached, i.e. it is isomorphic to $G_{2,6}$, whereas $H_6 \setminus \{6\}$ is isomorphic to $G_{2,5}$. Therefore, $J_{H_6}$ and $J_{H_6 \setminus \{6\}}$ are strongly unmixed by part \textbf{(a)}, hence $J_H$ is strongly unmixed.

\begin{figure}[ht!]
\begin{subfigure}[c]{0.45\textwidth}
\centering
\begin{tikzpicture}[scale=0.9]
\node[label={above:{\small $8$}}] (a) at (-1.43,1.96) {};
\node[label={right:{\small $7$}}] (b) at (0,1.5) {};
\node[label={right:{\small $6$}}] (c) at (0,0) {};
\node[label={below:{\small $5$}}] (d) at (-1.43,-0.46) {};
\node[label={above:{\small $4$}}] (e) at (-2.31,0.75) {};
\node[label={above:{\small $3$}}] (f) at (-3.81,0.75) {};
\node[label={above:{\small $2$}}] (g) at (-5.31,0.75) {};
\node[label={above:{\small $1$}}] (h) at (-6.81,0.75) {};
\draw (h) -- (g) -- (f) -- (e) -- (d) -- (c) -- (b) -- (a) -- (e)
(e) -- (c) -- (a) -- (d) -- (b) -- (e);
\end{tikzpicture}
\caption{The graph $G_{4,5}$} \label{F.G4,5StronglyUnmixed}
\end{subfigure}
\begin{subfigure}[c]{0.45\textwidth}
\centering
\begin{tikzpicture}[scale=0.9]
\node[label={right:{\small $1$}}] (a) at (1.5,0) {};
\node[label={below:{\small $2$}}] (b) at (0,0) {};
\node[label={below:{\small $3$}}] (c) at (-1.43,-0.46) {};
\node[label={left:{\small $4$}}] (d) at (-2.31,0.75) {};
\node[label={above:{\small $5$}}] (e) at (-1.43,1.96) {};
\node[label={above:{\small $6$}}] (f) at (0,1.5) {};
\node[label={right:{\small $7$}}] (g) at (1.5,1.5) {};
\draw (a) -- (b) -- (c) -- (d) -- (e) -- (f) -- (g)
(e) -- (b) -- (d) -- (f) -- (c)
(b) -- (f);
\end{tikzpicture}
\caption{A graph $H$ with $J_H$ strongly unmixed} \label{F.G44StronglyUnmixed}
\end{subfigure}
\caption{}
\end{figure}
\end{examples}

\begin{example}
If $G$ is a graph with $J_G$ strongly unmixed, not every vertex $v$ of $G$ produces $J_{G \setminus \{v\}}$, $J_{G_v}$, and $J_{G_v \setminus \{v\}}$ strongly unmixed. For example, let $G$ be the graph in Figure \ref{F.G92StronglyUnmixed}. We will show in Example \ref{E.G92StronglyUnmixed} that $J_G$ is strongly unmixed. This can also be done by applying Definition \ref{D.StronglyUnmixed} recursively, choosing $v=6$ at the first step, together with Example \ref{E.stronglyUnmixed} \textbf{(a)}. However, in Examples \ref{E.(G92_ragno)-vNotUnmixed} we showed that, if we consider the cut vertex $8$, then $J_{G \setminus \{8\}}$ is not unmixed.
\end{example}

\begin{remark}\label{R.conncomp}
If $G$ is a graph with connected components $G_1,\dots,G_c$ and $J_G$ strongly unmixed, then $J_{G_i}$ is strongly unmixed for every $i$. This can be easily shown by induction on the number of vertices of $G$ and then on the number of cut vertices of $G$. The converse also holds and it follows immediately from the definition.
\end{remark}

Now we prove the main result of this section, showing that the strong unmixedness of $J_G$ is sufficient for Cohen-Macaulayness.

\begin{theorem}\label{T.StronglyUnmixedImpliesCM}
Let $G$ be a graph. If $J_G$ is strongly unmixed, then $J_G$ is Cohen-Macaulay.
\end{theorem}

\begin{proof}
By Remarks \ref{R.connected} and \ref{R.conncomp}, we may assume that $G$ is connected. We proceed by induction on the number $n$ of vertices of $G$. If $n=2$, $G$ is a single edge, hence $J_G$ is Cohen-Macaulay. Fix $n>2$. We now use induction on the number $k \geq 0$ of cut vertices of $G$. If $k=0$, then $G$ is a complete graph by definition of strong unmixedness and, thus, $J_G$ is Cohen-Macaulay by Remark \ref{R.NoCutSetsComplete}.

Suppose that $k \geq 1$. Since $J_G$ is strongly unmixed, there exists a cut vertex $v$ of $G$ such that $J_{G \setminus \{v\}}$ is strongly unmixed and in particular unmixed. We consider the decomposition $J_G = A \cap B$, where
\[
A = \bigcap_{\substack{S \in \mc C(G)\\ v \notin S}} P_S(G) \quad \text{ and } \quad B = \bigcap_{\substack{S \in \mc C(G)\\ v \in S}} P_S(G),
\]
and the short exact sequence \eqref{Eq.shortExactSequenceFinal}. By Lemma \ref{L.completingNeighborsVertex} (1), $A = J_{G_v}$ is the binomial edge ideal of the graph $G_v$ which has $n$ vertices and $k-1$ cut vertices. Moreover, $J_{G_v}$ is strongly unmixed. Therefore, by induction on $k$, $A$ is Cohen-Macaulay and $\depth(R/A) = \dim(R/A) = n+1$.

By Proposition \ref{P.J_G-v_unmixed}, $B = (x_v,y_v) + J_{G \setminus \{v\}}$, where $J_{G \setminus \{v\}}$ is strongly unmixed and the graph $G \setminus \{v\}$ has less than $n$ vertices. Let $H_1$ and $H_2$ be the connected components of $G \setminus \{v\}$. By Remark \ref{R.conncomp}, $J_{H_1}$ and $J_{H_2}$ are strongly unmixed and, thus, Cohen-Macaulay by induction on $n$. In particular, $J_{G \setminus \{v\}}$ and $B$ are Cohen-Macaulay and $\depth(R/B) = \dim(R/B) = |V(G \setminus \{v\})|+2 = n+1$.

Finally, by Proposition \ref{P.J_G-v_unmixed}, $A+B = (x_v,y_v) + J_{G_v \setminus \{v\}}$. Recall that $J_{G_v \setminus \{v\}}$ is strongly unmixed and the graph $G_v \setminus \{v\}$ has $n-1$ vertices. By induction on $n$, it follows that $J_{G_v \setminus \{v\}}$ and $A+B$ are Cohen-Macaulay. In particular, $\depth(R/(A+B)) = \dim(R/(A+B)) = |V(G_v \setminus \{v\})|+1 = n$.

The Depth Lemma \cite[Lemma 3.1.4]{V94} applied to the short exact sequence \eqref{Eq.shortExactSequenceFinal} yields $\depth(R/J_G)  = n+1 = \dim(R/J_G)$.
\end{proof}

\begin{remark}
Theorem \ref{T.StronglyUnmixedImpliesCM} gives a new way of proving that a binomial edge ideal is Cohen-Macaulay. For instance, it follows that the ideal $J_H$ of the graph $H$ in Examples \ref{E.stronglyUnmixed} \textbf{(b)} is Cohen-Macaulay.
\end{remark}

Theorems \ref{T.StronglyUnmixedImpliesCM} and \ref{T.CMimpliesRCSP} together imply that if $J_G$ is strongly unmixed, then $G$ is accessible. The next result will allow us to show that also the reverse implication holds for some particular classes of graphs. This implies that Conjecture \ref{C.mainConj} holds for these graphs.

\begin{proposition}\label{P.RCSPImpliesStronglyUnmixed}
Let $\mc G$ be a class of accessible graphs such that for every $G \in \mc G$ either the connected components of $G$ are complete graphs or there exists a cut vertex $v$ of $G$ for which $G \setminus \{v\}, G_v, G_v \setminus \{v\} \in \mc G$.
Then, $J_G$ is strongly unmixed for every $G \in \mc G$. In particular, $J_G$ is Cohen-Macaulay.
\end{proposition}

\begin{proof} Let $G \in \mc G$. We proceed by induction on the number $n$ of vertices of $G$. If $n = 1$, there is nothing to prove. Fix $n \geq 2$ and let us proceed by induction on the number $k$ of cut vertices of $G$. If $G$ does not have cut vertices, then by Remark \ref{R.NoCutVerticesComplete} the connected components of $G$ are complete graphs and the claim follows. Let $k \geq 1$. Since $G \in \mc G$, there exists a cut vertex $v$ of $G$ such that $G \setminus \{v\}, G_v, G_v \setminus \{v\} \in \mc G$. Since $G \setminus \{v\}$ and $G_v \setminus \{v\}$ have $n-1$ vertices and $G_v$ has $n$ vertices and $k-1$ cut vertices, by induction it follows that $J_{G \setminus \{v\}}$, $J_{G_v \setminus \{v\}}$ and $J_{G_v}$ are strongly unmixed. Thus, $J_G$ is strongly unmixed by definition. The last part of the statement follows by Theorem \ref{T.StronglyUnmixedImpliesCM}.
\end{proof}

In order to use Proposition \ref{P.RCSPImpliesStronglyUnmixed}, in the next result we show that, if $G$ is accessible and $v$ is a cut vertex, the accessibility of $G \setminus \{v\}$ is equivalent to the unmixedness of $J_{G \setminus \{v\}}$.

\begin{proposition} \label{P.J_G-v_unmixed2}
Let $v$ be a cut vertex of a connected accessible graph $G$. The following statements are equivalent:
\begin{enumerate}
\item[$\mathrm{(1)}$] $J_{G \setminus \{v\}}$ is unmixed;
\item[$\mathrm{(2)}$] $\mathcal{C}(G \setminus \{v\})$ is an accessible set system.
\end{enumerate}
In particular, if one of the above conditions holds, then $G \setminus \{v\}$ is accessible.
\end{proposition}

\begin{proof}
Let $H=G \setminus \{v\}$ and let $H_1$ and $H_2$ be the connected components of $H$. Assume that $J_{H}$ be unmixed. By Proposition \ref{P.J_G-v_unmixed}, we have
\[
\mc C(H)=\{S \subseteq V(H) : S \cup \{v\} \in \mc C(G)\}.
\]
Let $S \in \mc C(H)$. If all the elements of $S$ are cut vertices of $G$, then the same is true for $S \cup \{v\}$. Since $J_G$ is unmixed, by Lemma \ref{L.cutSetsConsistingOfCutVertices}, $(S \setminus \{s\}) \cup \{v\} \in \mc C(G)$ for every $s \in S$. On the other hand, if $S$ contains a non-cut vertex $s$, by Proposition \ref{P.removingNonCutVertex2} we have that $(S \setminus \{s\}) \cup \{v\} \in \mc C(G)$.

Conversely, assume by contradiction that $J_H$ is not unmixed. By Proposition \ref{P.J_G-v_unmixed}, we may assume that there exists $T \in \mc C(H)$ such that $N_{H_1}(v) \subseteq T$.
We notice that $T'=T \cap V(H_1)$ is a cut set of $H$ by Remark \ref{R.connected} (1). By assumption, there exists $t_1 \in T'$ such that $T' \setminus \{t_1\} \in \mc C(H)$. If $t_1 \in N_{H_1}(v)$, then set $U=T'$; otherwise by assumption there exists $t_2 \in T' \setminus \{t_1\}$ such that $T' \setminus \{t_1,t_2\} \in \mc C(H)$. If $t_2 \in N_{H_1}(v)$, then set $U=T' \setminus \{t_1\}$; otherwise we keep removing elements from $T'$ until we find $t_r \in N_{H_1}(v)$ such that $T' \setminus \{t_1,\dots,t_r\} \in \mc C(H)$ and we set $U=T' \setminus \{t_1,\dots,t_{r-1}\}$, where $t_1,\dots,t_{r-1} \notin N_{H_1}(v)$. Moreover, since $U \in \mc C(H)$, for every $u \in U$ we have
\[
c_G(U\setminus\{u\}) \leq c_H(U\setminus\{u\}) < c_H(U) = c_{G}(U)
\]
because $N_{H_1}(v) \subseteq U \subseteq V(H_1)$. Hence, $U$ is also a cut set of $G$. Furthermore, for every $u \in U \setminus \{t_r\}$ we have
\[
c_G((U\cup \{v\})\setminus\{t_r,u\}) = c_H(U\setminus\{t_r,u\}) < c_H(U \setminus \{t_r\}) = c_{G}((U\cup\{v\})\setminus \{t_r\})
\]
and  $c_G(U \setminus \{t_r\})<c_G((U \cup \{v\}) \setminus \{t_r\})$ since $v$ is adjacent to $t_r$ and to some vertex of $H_2$. It follows that $(U \cup \{v\})\setminus \{t_r\}$ is a cut set of $G$. Finally, since $U$ and $(U \cup \{v\})\setminus \{t_r\}$ are cut sets of $G$, $J_G$ is unmixed, and $N_{H_1}(v) \subseteq U \subseteq V(H_1)$, we have
\[
|U|+1=c_G((U \cup \{v\})\setminus \{t_r\})=c_H(U \setminus \{t_r\}) < c_H(U)=c_G(U)=|U|+1,
\]
which yields a contradiction.
\end{proof}

\begin{example}
In general, it is possible that the equivalent conditions of Proposition \ref{P.J_G-v_unmixed2} hold, but $G$ is not accessible. For instance, if $G$ is the graph in Figure \ref{F.A6}, then $G \setminus \{6\}$ satisfies both conditions of Proposition \ref{P.J_G-v_unmixed2}, but $G$ is not accessible. In fact, $\{3,4\} \in \mc C(G)$, but neither $3$ nor $4$ are cut vertices of $G$.
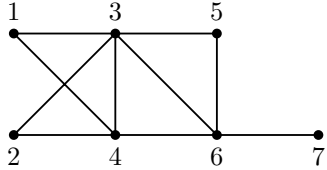
\begin{figure}[ht!]
\centering
\begin{tikzpicture}[scale=0.9]
\node[label={below:{\small $2$}}] (a) at (0,0) {};
\node[label={above:{\small $1$}}] (b) at (0,1.5) {};
\node[label={below:{\small $4$}}] (c) at (1.5,0) {};
\node[label={above:{\small $3$}}] (d) at (1.5,1.5) {};
\node[label={below:{\small $6$}}] (e) at (3,0) {};
\node[label={above:{\small $5$}}] (f) at (3,1.5) {};
\node[label={below:{\small $7$}}] (g) at (4.5,0) {};
\draw (4.5,0) -- (3,0) -- (3,1.5) -- (1.5,1.5) -- (3,0) -- (1.5,0) -- (1.5,1.5) -- (0,1.5) -- (1.5,0) -- (0,0) -- (1.5,1.5);
\end{tikzpicture}
\caption{A graph $G$ such that $J_{G_6}$ is not strongly unmixed} \label{F.A6}
\end{figure}
\end{example}

\begin{corollary}\label{C.rcsp2}
Let $G$ be an accessible graph and $v$ be a cut vertex of $G$ such that $J_{G \setminus \{v\}}$ is unmixed. Then, $G \setminus \{v\}$, $G_v$, and $G_v \setminus \{v\}$ are accessible.
\end{corollary}

\begin{proof}
The graphs $G_v$ and $G \setminus \{v\}$ is accessible by Lemma \ref{L.completingNeighborsVertex} (3) and Proposition \ref{P.J_G-v_unmixed2}.

By Lemma \ref{L.Gv-vHasRCSP}, $J_{G_v \setminus \{v\}}$ is unmixed and $$\mc C(G_v \setminus \{v\})=\mc C(G_v) \setminus \{S \in \mc C(G_v) : N_G(v) \subseteq S\}.$$ Let $S \in \mc C(G_v \setminus \{v\}) \subseteq \mc C(G_v)$. Then, there exists $s \in S$ such that $S \setminus \{s\} \in \mc C(G_v)$. Clearly, $N_G(v) \not \subseteq S \setminus \{s\}$, since $S \in \mc C(G_v \setminus \{v\})$ and, hence, $S \setminus \{s\} \in \mc C(G_v \setminus \{v\})$.
\end{proof}

As a consequence of Proposition \ref{P.RCSPImpliesStronglyUnmixed} and the previous corollary, if we set $\mc G$ to be the class of all accessible graphs, an affirmative answer to the following question would completely settle Conjecture \ref{C.mainConj}.

\begin{question}\label{Q.GoodCutVertex}
If $G$ is a connected non-complete accessible graph, does there exist a cut vertex $v$ of $G$ such that $J_{G \setminus \{v\}}$ is unmixed?
\end{question}

In the next section, we are going to prove that the answer is positive for chordal and traceable graphs.

\section{Cohen-Macaulayness of chordal and traceable graphs}
\label{S.Chordal}

The main goal of this section is to prove that for every chordal or traceable graphs $G$, being accessible is equivalent to both Cohen-Macaulayness and strong unmixedness of $J_G$.

We start by recalling the notion of block graph. A connected subgraph of $G$ that cannot be disconnected by removing a vertex and is maximal with respect to this property is called a \textit{block} of $G$. The \textit{block graph} of $G$, denoted by $\mc B(G)$, is a graph whose vertices are the blocks of $G$ and such that there is an edge between two vertices if and only if the corresponding blocks contain a common cut vertex of $G$. In \cite[Proposition 1.3]{R19}, Rinaldo proves that the block graph of any connected graph $G$ with $J_G$ unmixed is a tree.

To answer Question \ref{Q.GoodCutVertex}, the next result allows us to focus only on one block instead of the whole graph.

\begin{proposition}\label{P.goodCutVertexUnmixed}
Let $G$ be a connected graph with $J_G$ unmixed and suppose that for any block $B$ of $G$ there exists a cut vertex $v_B$ of $G$ in $V(B)$ such that there are no cut sets of $G\setminus \{v_B\}$ containing $N_B(v_B)$. Then, there exists a cut vertex $v$ of $G$ such that $J_{G \setminus \{v\}}$ is unmixed.
\end{proposition}

\begin{proof}
Recall that $\mc B(G)$ is a tree by \cite[Proposition 1.3]{R19}, and every cut vertex of $G$ belongs to exactly two blocks of $G$ because $J_G$ is unmixed. Let $B_1$ be a block corresponding to a leaf of $\mc B(G)$. Therefore, there is a unique cut vertex $v_1$ of $G$ in $V(B_1)$. By assumption, there are no cut sets of $G\setminus \{v_1\}$ containing $N_{B_1}(v_1)$. Let $B_2$ be the other block of $G$ containing $v_1$. Again by assumption, there is a cut vertex $v_2 \in V(B_2)$ such that $N_{B_2}(v_2)$ is not contained in any cut sets of $G\setminus \{v_2\}$. If $v_1=v_2$, then $J_{G\setminus \{v_1\}}$ is unmixed by Proposition \ref{P.J_G-v_unmixed}. Otherwise, we consider the other block $B_3$ containing $v_2$ and the cut vertex $v_3$ given by the assumption. We can continue in this way and, if we do not find any cut vertex $v_i$ for which $J_{G\setminus \{v_i\}}$ is unmixed, after finitely many steps we reach a block $B_p$ that is a leaf in the block graph, because $\mc B(G)$ does not contain cycles. Hence, there is a unique cut vertex of $G$ in $V(B_p)$, which is $v_{p-1}$, the only element of $V(B_{p-1}) \cap V(B_p)$. Thus, $J_{G\setminus \{v_{p-1}\}}$ is unmixed.
\end{proof}

To prove that chordal accessible graphs satisfy the condition of Proposition \ref{P.goodCutVertexUnmixed} we first need a technical result.

\begin{lemma} \label{L.NeighborsAdjCutVertices}
Let $G$ be a connected graph such that $J_G$ is unmixed and let $B$ be a block of $G$. Let $v_1$ and $v_2$ be two cut vertices of $G$ belonging to $B$.
\begin{itemize}
\item[{\rm (1)}] If $v_1$ and $v_2$ are adjacent and there exists a cut set $S_i \in \mc C(G \setminus \{v_i\})$ such that $N_B(v_i) \subseteq S_i$ for every $i=1,2$, then $N_B(v_1) \nsubseteq N_B[v_2]$ and $N_B(v_2) \nsubseteq N_B[v_1]$.
\item[{\rm (2)}] If $N_B[v_1] = V(B)$, then there exists a cut vertex $w \in B$ of $G$ such that $N_B(w) \nsubseteq S$ for every $S \in \mc C(G \setminus \{w\})$.
\end{itemize}
\end{lemma}

\begin{proof}
{\rm (1)} Assume by contradiction that $N_B(v_1) \subseteq N_B[v_2]$. Let $C$ and $C_B$ be the connected components of $G \setminus \{v_1\}$, where $C_B$ contains $B \setminus \{v_1\}$. Since $N_B(v_1) \subseteq N_B[v_2] \subseteq S_2 \cup \{v_2\}$ and $v_1 \in S_2$, it follows that $v_1$ reconnects at least two connected components of $C \setminus S_2$.

Set $T_B = S_1 \cap C_B$, which is a cut set of $C_B$ and also of $G$ because it contains $N_B(v_1)$. On the other hand, the set $W = S_2 \cap (V(C) \cup \{v_1\})$ is a cut set of $G[V(C) \cup \{v_1\}]$ by Remark \ref{R.connected}: indeed, every $w \in W \setminus \{v_1\}$ reconnects at least two connected components of $G[V(C) \cup \{v_1\}] \setminus S_2$ because $v_1 \in S_2$; this is also true for $v_1$ as explained in the beginning. In particular, $W$ is a cut set of $G$.

Now, we notice that $T_B \sqcup W \in \mc C(G)$ because $T_B$ is a cut set of $C_B$ and $W$ is a cut set of $G[V(C) \cup \{v_1\}]$ containing $v_1$. Moreover,
\[
c_G(T_B \sqcup W) = (c_G(T_B)-1) + (c_G(W)-1) = |T_B| + |W| = |T_B \sqcup W|,
\]
where the second equality holds because $J_G$ is unmixed. Thus, the equality $c_G(T_B \sqcup W) = |T_B \sqcup W|$ contradicts the unmixedness of $J_G$. \\
{\rm (2)} We first suppose that $v_1$ is the only cut vertex of $G$ in $V(B)$. Then, $N_B(v_1) = V(B) \setminus \{v_1\}$ and $N_B(v_1) \nsubseteq S$ for every $S \in \mc C(G \setminus \{v_1\})$, otherwise every element of $N_B(v_1)$ would not be adjacent to any vertex of $G \setminus (\{v_1\} \cup S)$. Suppose now that there is another cut vertex $w \in V(B)$ of $G$, $w \neq v_1$. Then, by assumption $N_B(w) \subseteq N_B[v_1] = V(B)$ and this contradicts {\rm (1)}.
\end{proof}

Recall that a graph $G$ is called \textit{chordal} if all its induced cycles have length three.

\begin{proposition} \label{P.GoodCutVertexChordal}
Let $G$ be a non-complete connected chordal accessible graph. Then, there exists a cut vertex $v$ of $G$ such that $J_{G\setminus \{v\}}$ is unmixed.
\end{proposition}

\begin{proof}
Let $B$ be a block of $G$. By Proposition \ref{P.goodCutVertexUnmixed}, it is enough to show the following statement.\\[1mm]
\textbf{Claim:} there exists a cut vertex $v$ of $G$ in $V(B)$ such that $N_{B}(v)$ is not contained in any cut set of $G\setminus \{v\}$.
\\[1mm] \indent By Remark \ref{R.NoCutVerticesComplete}, $G$ has at least one cut vertex. In particular, $V(B)$ contains at least one cut vertex of $G$. Let $v_1, \dots, v_r$ be the cut vertices of $G$ belonging to $V(B)$, for some $r \geq 1$.

If $r=1$, then by Proposition \ref{P.EveryVertexIsAdjacentToACutVertex} $N_B[v_1]=V(B)$ and the claim follows by Lemma \ref{L.NeighborsAdjCutVertices} {\rm (2)}. Hence, we may assume $r \geq 2$.

Assume now that $N_B[v_1] \subsetneq V(B)$, otherwise we conclude again by Lemma \ref{L.NeighborsAdjCutVertices} {\rm (2)}. We want to find a cut vertex $v$ of $G$ in $V(B)$ fulfilling the claim. If there is a cut vertex $w$ such that $N_B(w) \cup N_B[v_{1}] =V(B)$, then we set $v=w$. Otherwise, by Proposition \ref{P.subgraphCutVerticesConnected} we can choose $v_{i_2}$ adjacent to $v_{1}$ and we have $N_B[v_{1}] \cup N_B[v_{i_2}] \subsetneq V(B)$. Again, if there exists a cut vertex $w$ such that $N_B(w) \cup N_B[v_{1}] \cup N_B[v_{i_2}] =V(B)$, then we set $v=w$. If not, by Proposition \ref{P.subgraphCutVerticesConnected}, we can continue in this way choosing at each step a cut vertex $v_{i_j}$ adjacent to at least one of $v_1,v_{i_2},\dots,v_{i_{j-1}}$. By Propositions \ref{P.EveryVertexIsAdjacentToACutVertex}, we eventually find a cut vertex $v$ of $G$ in $V(B)$ such that $N_B(v) \cup N_B[v_{1}] \cup N_B[v_{i_2}] \cup \dots \cup N_B[v_{i_c}] =V(B)$ and $N_B[v_{1}] \cup N_B[v_{i_2}] \cup \dots \cup N_B[v_{i_c}] \subsetneq V(B)$; in particular, $v$ is adjacent to at least one of $v_1,v_{i_2},\dots,v_{i_c}$. Moreover, the subgraph induced by $G$ on $\{v_{1}, v_{i_2}, \dots, {v_{i_c}}\}$ is connected. To simplify the notation, we assume without loss of generality that $\{v_1,v_{i_2}, \dots, v_{i_c}\}=\{v_1, \dots, v_c\}$.

We set $N=N_B[v_1] \cup N_B[v_2] \cup \dots \cup N_B[v_c]$ and note that $V(B)=N \cup N_B(v)$; in particular, $v \in N$ and $N_B(v) \setminus N \neq \emptyset$.

We now assume by contradiction that $N_B(v)$ is contained in a cut set $S$ of $G \setminus \{v\}$. In particular, $N_B(v)$ is a cut set of $G$ by Remark \ref{R.NeighborsCutSet}.

Assume  first that there exists $x \in N_B(v) \setminus N$ which is not a cut vertex of $G$.
Since $N_B(v)$ is a cut set containing $x$ and $x$ is not a cut vertex of $G$, it follows that $x$ is adjacent to a vertex $w \in V(B) \setminus N_B[v] \subseteq N \setminus \{v\}$.
Thus, $w \in N_B(v_i)$ for some $i=1, \dots, c$. We also know that $v \in N$, therefore $v \in N_B(v_j)$ for some $j=1, \dots, c$. Consider a minimal path $v_i=u_1, u_2, \dots, u_a=v_j$ in $G$, where $a \geq 1$ and $u_k \in \{v_1, \dots, v_c\}$ for every $k$, which exists because $G[\{v_1,\dots,v_c\}]$ is connected. Let $p$ be the maximum index for which $\{w,u_p\} \in E(G)$ and $q$ the minimum index greater than or equal to $p$ such that $\{u_q,v\} \in E(G)$. Hence, in $G$ there is an induced cycle $v, x, w, u_p, \dots, u_q, v$. Thus, its length is at least four and it has no chords since $x \notin N$ and $w \notin N_B[v]$, against $G$ being chordal.

It remains to consider the case in which $N_B(v) \setminus N$ contains only cut vertices of $G$. Let $z \in N_B(v) \setminus N$. We first show that $N_B(z) \subseteq N_B[v]$. Suppose that there exists $w \in N_B(z) \setminus N_B[v]$. Then, $w \in N \setminus \{v_1,\dots,v_c\}$ and $\{w,v_i\} \in E(G)$ for some $i \in \{1,\dots,c\}$. Consider a minimal path $w, v_i=u_0, u_1,\dots, u_a, v$, where $u_k \in \{v_1,\dots,v_c\}$ for every $k$. As before, let $p$ be the maximum index for which $\{w,u_p\} \in E(G)$ and $q$ the minimum index greater than or equal to $p$ such that $\{u_q,v\} \in E(G)$. Thus, there is an induced cycle $v, z, w, u_p, \dots, u_q, v$, where $z$ is not adjacent to any $u_k$ since $z \notin N$. Hence, since $G$ is chordal, we have that $\{v,w\} \in E(G)$, a contradiction. Therefore, $N_B(z) \subseteq N_B[v]$.

We may assume that there exists $T \in \mc C(G \setminus \{z\})$ such that $N_B(z) \subseteq T$, otherwise $z$ would satisfy the claim at the beginning of the proof. This contradicts Lemma \ref{L.NeighborsAdjCutVertices} {\rm (1)}.
\end{proof}

We are ready to prove Conjecture \ref{C.mainConj} for chordal graphs.

\begin{theorem} \label{T.FinalChordal}
If $G$ is a chordal graph, then the following conditions are equivalent:
\begin{itemize}
\item[{\rm (1)}] $J_G$ is Cohen-Macaulay;
\item[{\rm (2)}] $J_G$ is strongly unmixed;
\item[{\rm (3)}] $G$ is accessible.
\end{itemize}
\end{theorem}

\begin{proof}
By Remarks \ref{R.connected} and Remark \ref{R.conncomp}, it is enough to prove the claim for $G$ connected. By Theorems \ref{T.CMimpliesRCSP} and \ref{T.StronglyUnmixedImpliesCM}, we only need to prove that {\rm (3)} implies {\rm (2)}. We note that if $G$ is chordal and $v$ is a cut vertex of $G$, then $G_v$, $G \setminus \{v\}$, and $G_v \setminus \{v\}$ are also chordal. Setting $\mc G$ to be the class of chordal accessible graphs, the claim follows by Proposition \ref{P.GoodCutVertexChordal}, Corollary \ref{C.rcsp2}, and Proposition \ref{P.RCSPImpliesStronglyUnmixed}.
\end{proof}

Next we prove that the three conditions in Theorem \ref{T.FinalChordal} are equivalent also for another large class of graphs.
We recall that a connected graph $G$ is called \textit{traceable} if it contains a \textit{Hamiltonian path}, i.e., a path that visits each vertex of $G$ exactly once. With a slight abuse of notation, we say that a disconnected graph is \textit{traceable} if each of its connected components contains a Hamiltonian path.

\begin{lemma}\label{L.Blocksoftraceable}
If $G$ is a traceable graph, then every block of $G$ contains at most two cut vertices of $G$. Moreover, if $G$ is accessible and a block contains two cut vertices of $G$, then they are adjacent.
\end{lemma}

\begin{proof}
Clearly, it is enough to consider the case in which $G$ is connected. Suppose that there is a block $B$ of $G$ containing three cut vertices $v_1,v_2,v_3$ and let $B_i$ a block containing $v_i$ different from $B$. It is easy to see that any path that visits all vertices of $B,B_1,B_2,B_3$ has to pass at least twice through one of the $v_i$'s, and hence such a path is not Hamiltonian, against the assumption. The last part of the statement now follows by Proposition \ref{P.subgraphCutVerticesConnected}.
\end{proof}

As in the case of chordal graphs, for traceable graphs we need to find a cut vertex $v$ such that $J_{G \setminus \{v\}}$ is unmixed. In the next proposition we prove a more general result, which could be useful to answer Question \ref{Q.GoodCutVertex}.

\begin{proposition}\label{P.CompleteSubgraphOnCutVertices}
Let $G$ be a connected non-complete accessible graph. Assume that for every block $B$ of $G$ the subgraph induced by $G$ on the cut vertices of $G$ belonging to $V(B)$ is complete. Then, there exists a cut vertex $v$ of $G$ such that $J_{G \setminus \{v\}}$ is unmixed.
\end{proposition}

\begin{proof}
Let $B$ be a block of $G$ and let $v_1, \dots, v_r$ be the cut vertices of $G$ belonging to $V(B)$, where $r \geq 1$, by Remark \ref{R.NoCutVerticesComplete}.

By Proposition \ref{P.goodCutVertexUnmixed}, it is enough to show that there exists a cut vertex $v \in B$ of $G$ such that $N_{B}(v)$ is not contained in any cut set of $G\setminus \{v\}$.

We may assume that $N_B[v_i] \subsetneq V(B)$ for every $i=1, \dots, r$, by Lemma \ref{L.NeighborsAdjCutVertices} {\rm (2)}; in particular, by assumption, $\{v_1,\dots,v_r\} \subsetneq V(B)$. Given $w \in V(B)\setminus \{v_1, \dots, v_r\}$, we define $c(w)$ to be the number of cut vertices of $G$ adjacent to $w$, i.e., $c(w)=|N_B(w) \cap \{v_1, \dots, v_r\}|$. Consider $u \in V(B)\setminus \{v_1, \dots, v_r\}$ for which $c(u)$ is minimal. By Proposition \ref{P.EveryVertexIsAdjacentToACutVertex}, $c(u)>0$ and, without loss of generality, suppose that $u$ is adjacent to $v_1, \dots, v_c$, where $c=c(u)$. Moreover, we can assume $c<r$, otherwise $B$ would be complete, thus $N_B[v_i]=V(B)$ for every $i$.

Define $N=N_B[v_{c+1}] \cup N_B[v_{c+2}] \cup \dots \cup N_B[v_r] \neq \empty V(B)$ and $N'=\{w \in N \mid N_B(w) \nsubseteq N\} \subseteq (N \setminus \{v_{c+1}, \dots, v_r\})$. As in Proposition \ref{P.EveryVertexIsAdjacentToACutVertex}, it is easy to see that $N'$ is a cut set of $G$.
Moreover, if $w \in V(B)$ and $w \notin N_B[v_1]$, then $w \in N$ because it is not a cut vertex and $c(w) \geq c(u)=c$. Therefore, $V(B)=N \cup N_B[v_1]$.

Now we assume by contradiction that $N_B(v_1) \subseteq S$ for some $S \in \mc C(G \setminus \{v_1\})$.
Clearly, $S$ is also a cut set of $G$ and by Corollary \ref{C.rcsp} we can order the elements of $S=\{z_1, \dots, z_t\}$ in such a way that $\{z_1, \dots, z_i\} \in \mc C(G)$ for every $i=1, \dots, t$.
Let $N_B(v_1) \setminus N=\{z_{i_1}, \dots, z_{i_s}\}$ with $i_1< \dots < i_s$, which is not empty because it contains $u$. Since $S$ is a cut set of $G \setminus \{v_1\}$, $z_{i_s}$ is adjacent to at least two vertices $a_s, a_{s+1} \in (V(B)\setminus \{v_1\}) \setminus S \subseteq V(B) \setminus N_B[v_1]\subseteq N \setminus \{v_1, \dots, v_r\}$, and both $a_s$ and $a_{s+1}$ are in $N'$ because $z_{i_s}\notin N$.

Consider now $S'=\{z_1, \dots, z_{i_{s-1}}\}$, which is a cut set of $G$. In $G \setminus S'$ there is a connected component containing $v_1, z_{i_s}, a_s$, and $a_{s+1}$. Therefore, $z_{i_{s-1}}$ has to be adjacent to a vertex\break $a_{s-1} \in V(B) \setminus (N_B[v_1] \cup \{a_s,a_{s+1}\}) \subseteq N \setminus \{v_1,\dots,v_r,a_s, a_{s+1}\}$. Again, $a_{s-1} \in N' \setminus \{a_s,a_{s+1}\}$ and it is not a cut vertex of $G$. Repeating the same argument, we find $\{a_1, \dots, a_{s+1}\} \subseteq N'$ where all the $a_i$'s are distinct and are not cut vertices of $G$.

Moreover, $v_1, \dots, v_c \in N'$ because they are adjacent to $v_r$ by assumption and to $u$ by construction.
Hence, $|N'| \geq c+s+1$ and the connected components of $G \setminus N'$ are the following:
\begin{itemize}
\item the component containing $N\setminus N'$, which is connected because $G[\{v_{c+1}, \dots,v_r\}]$ is connected (it is indeed complete) by assumption;
\item the $c$ components outside $B$, each obtained by removing $v_i$, for $i=1,\dots,c$;
\item the components of $N_B(v)\setminus N$, which are at most $s$ because $N_B(v)\setminus N$ has cardinality $s$.
\end{itemize}
Since $N'$ is a cut set of $G$ and $J_G$ is unmixed, we get $c_G(N')\leq 1+c+s \leq |N'|=c_G(N')-1$, which yields a contradiction.
\end{proof}

By Lemma \ref{L.Blocksoftraceable} and Proposition \ref{P.CompleteSubgraphOnCutVertices}, we get the following consequence.

\begin{corollary} \label{C.cutVertexTraceable}
If $G$ is a non-complete traceable accessible graph, then it contains a cut vertex $v$ such that $J_{G \setminus \{v\}}$ is unmixed.
\end{corollary}

If $G$ is traceable and $v$ is a cut vertex of $G$, clearly also $G_v$, $G \setminus \{v\}$, and $G_v \setminus \{v\}$ are traceable. Then, in light of Corollary \ref{C.cutVertexTraceable}, we can prove the next result with the same argument used for Theorem \ref{T.FinalChordal}.

\begin{theorem} \label{T.traceableCM}
If $G$ is a traceable graph, then the following conditions are equivalent:
\begin{itemize}
\item[{\rm (1)}] $J_G$ is Cohen-Macaulay;
\item[{\rm (2)}] $J_G$ is strongly unmixed;
\item[{\rm (3)}] $G$ is accessible.
\end{itemize}
\end{theorem}

As a consequence of Theorem \ref{T.traceableCM}, we recover the equivalence (a) $\Leftrightarrow$ (d) in \cite[Theorem 6.1]{BMS18}.

\begin{corollary}\label{C.bipartiteCM}
If $G$ is a bipartite graph, then the following conditions are equivalent:
\begin{itemize}
\item[{\rm (1)}] $J_G$ is Cohen-Macaulay;
\item[{\rm (2)}] $J_G$ is strongly unmixed;
\item[{\rm (3)}] $G$ is accessible.
\end{itemize}
\end{corollary}

\begin{proof}
We only need to prove that {\rm (3)} implies {\rm (2)}. Notice that every bipartite accessible graph is traceable: this follows by the explicit description of such graphs given in \cite[Theorem 6.1 (c)]{BMS18}. In fact, every block of such a graph has exactly two cut vertices and is traceable. The claim now follows by Theorem \ref{T.traceableCM}.
\end{proof}

\begin{example}\label{E.traceable}
The graph $H$ in Figure \ref{F.traceable} is traceable and accessible, but non-bipartite and non-chordal. This graph already appeared in the classification of Cohen–Macaulay bicyclic graphs \cite[Lemma 3.2, Figure 7]{R19}. The cut sets of $H$ are $\mc C(H) = \{\emptyset, \{2\}, \{6\}, \{2, 6\}, \{2, 4\}, \{4, 6\}, \{3, 6\}, \{2, 4, 6\}\}$ and it easy to show that $H$ is accessible. Hence, by Theorem \ref{T.traceableCM}, $J_H$ is Cohen-Macaulay and strongly unmixed.

\begin{figure}[ht!]
\centering
\begin{tikzpicture}[scale=0.9]
\node[label={left:{\small $1$}}] (a) at (0,3) {};
\node[label={above left:{\small $2$}}] (b) at (0,1.5) {};
\node[label={below left:{\small $3$}}] (c) at (0,0) {};
\node[label={below right:{\small $4$}}] (d) at (1.5,0) {};
\node[label={right:{\small $5$}}] (e) at (2.8,0.75) {};
\node[label={above right:{\small $6$}}] (f) at (1.5,1.5) {};
\node[label={right:{\small $7$}}] (g) at (1.5,3) {};
\draw (0,3) -- (0,1.5) -- (0,0) -- (1.5,0) -- (2.8,0.75) -- (1.5,1.5) -- (1.5,3)
(0,1.5) -- (1.5,1.5) -- (1.5,0);
\end{tikzpicture}
\caption{A traceable graph $H$} \label{F.traceable}
\end{figure}
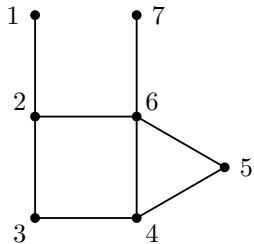
\end{example}

\begin{example}\label{E.G92StronglyUnmixed}
Theorem \ref{T.StronglyUnmixedImpliesCM} can be useful to find new examples of graphs, which are not chordal nor traceable, and whose binomial edge ideal is Cohen-Macaulay. For instance, the graph $G$ in Figure \ref{F.G92StronglyUnmixed} is not chordal nor traceable (since it has a block containing three cut vertices of $G$), but $J_G$ is strongly unmixed. In fact, if we consider the cut vertex $8$, then
\begin{itemize}
\item $G \setminus \{8\}$ is traceable and $\mc C(G \setminus \{8\}) = \{\emptyset, \{2\}, \{4\}, \{7\}, \{2, 4\}, \{2, 5\}, \{2, 7\}, \{3, 4\}, \{2, 5, 7\}\}$;
\item $G_8$ is chordal and $\mc C(G_8) = \{\emptyset, \{2\}, \{7\}, \{2, 7\}, \{2, 5, 7\}\}$;
\item $G_8 \setminus \{8\}$ is chordal and $\mc C(G_8 \setminus \{8\}) = \{\emptyset, \{2\}, \{7\}, \{2, 7\}, \{2, 5, 7\}\}$.
\end{itemize}
It is straightforward to check that the previous three graphs are accessible. Hence, by Theorems \ref{T.FinalChordal} and \ref{T.traceableCM}, their binomial edge ideals are strongly unmixed. We conclude that $J_G$ is Cohen-Macaulay by Theorem \ref{T.StronglyUnmixedImpliesCM}.
\end{example}

\section{Further remarks and problems}

Finally, we discuss some examples and open problems. First of all, we notice that it is enough to prove Conjecture \ref{C.mainConj} for \textit{indecomposable} graphs, see \cite[Definition 2.1]{R19}, i.e., graphs that cannot be decomposed as $G = G_1 \cup G_2$ where $V(G_1) \cap V(G_2) = \{v\}$ for some vertex $v$ which is free both in $G_1$ and in $G_2$. In fact, if $G$ is decomposable as $G = G_1 \cup G_2$, then
\begin{itemize}
\item $J_G$ is Cohen-Macaulay if and only if $J_{G_1}$, $J_{G_2}$ are Cohen-Macaulay by \cite[Theorem 2.7]{RR14} and,
\item using the description of the cut sets in \cite[Lemma 2.3]{RR14}, it is easy to show that $G$ is accessible if and only if $G_1,G_2$ are accessible.
\end{itemize}

Given two graphs $G$ and $H$, we introduce a new construction that produces a new graph obtained by gluing certain subgraphs of $G$ and $H$ along a cut vertex. We illustrate it through an example.

\begin{example}\label{E.generalized2b}
Let us consider the graphs $G$ and $H$ in Figure \ref{F.chordal2b} and \ref{F.traceable2b} respectively.

\begin{figure}[ht!]
\begin{subfigure}[c]{0.3\textwidth}
\centering
\begin{tikzpicture}
\node[label={below:{\small $1$}}] (a) at (-1.5,0) {};
\node[label={below left:{\small $2$}}] (b) at (-3,0) {};
\node[label={above left:{\small $3$}}] (c) at (-3,1.5) {};
\node[label={above left:{\small $4$}}] (d) at (-1.5,1.5) {};
\node[label={above:{\small $5$}}] (e) at (-1.5,3) {};
\node[label={above left:{\small $6$}}] (f) at (-3,3) {};
\node[label={above right:{\small $7$}}] (g) at (0,3) {};
\node[label={above right:{\small $8$}}] (h) at (0,1.5) {};
\node[label={below right:{\small $9$}}] (i) at (0,0) {};
\draw (-1.5,0) -- (-3,0) -- (-3,1.5) -- (-1.5,1.5) -- (-1.5,3) -- (-3,0) -- (0,1.5) -- (0,3) -- (-1.5,1.5) -- (-3,0)
(-1.5,1.5) -- (0,1.5)
(-3,3) -- (-1.5,3) -- (0,3)
(-1.5,3) -- (0,1.5) -- (0,0);
\end{tikzpicture}
\caption{A chordal graph $G$} \label{F.chordal2b}
\end{subfigure}
\begin{subfigure}[c]{0.3\textwidth}
\centering
\begin{tikzpicture}
\node[label={above left:{\small $10$}}] (g) at (0,3) {};
\node[label={above left:{\small $11$}}] (h) at (0,1.5) {};
\node[label={below left:{\small $12$}}] (i) at (0,0) {};
\node[label={below right:{\small $13$}}] (j) at (1.5,0) {};
\node[label={right:{\small $14$}}] (k) at (2.8,0.75) {};
\node[label={above right:{\small $15$}}] (l) at (1.5,1.5) {};
\node[label={above right:{\small $16$}}] (m) at (1.5,3) {};
\draw (0,3) -- (0,1.5) -- (0,0) -- (1.5,0) -- (2.8,0.75) -- (1.5,1.5) -- (1.5,3)
(0,1.5) -- (1.5,1.5) -- (1.5,0);
\end{tikzpicture}
\caption{A traceable graph $H$} \label{F.traceable2b}
\end{subfigure}
\begin{subfigure}[c]{0.38\textwidth}
\centering
\begin{tikzpicture}
\node[label={below right:{\phantom{\small $16$}}}] (a) at (-1.5,0) {};
\node (b) at (-3,0) {};
\node (c) at (-3,1.5) {};
\node (d) at (-1.5,1.5) {};
\node (e) at (-1.5,3) {};
\node (f) at (-3,3) {};
\node[label={above right:{\phantom{\small $16$}}}] (g) at (0,3) {};
\node[label={above right:{\small $v$}}] (h) at (0,1.5) {};
\node (i) at (0,0) {};
\node (j) at (1.5,0) {};
\node (k) at (2.8,0.75) {};
\node (l) at (1.5,1.5) {};
\node (m) at (1.5,3) {};
\draw (-1.5,0) -- (-3,0) -- (-3,1.5) -- (-1.5,1.5) -- (-1.5,3) -- (-3,0) -- (0,1.5) -- (0,3) -- (-1.5,1.5) -- (-3,0)
(-1.5,1.5) -- (0,1.5)
(-3,3) -- (-1.5,3) -- (0,3)
(-1.5,3) -- (0,1.5) -- (0,0)
(0,0) -- (1.5,0) -- (2.8,0.75) -- (1.5,1.5) -- (1.5,3)
(0,1.5) -- (1.5,1.5) -- (1.5,0);
\end{tikzpicture}
\caption{The graph $F$} \label{F.2bChordalTraceable}
\end{subfigure}
\caption{}\label{F.generalized2b}
\end{figure}

Notice that $J_{G \setminus \{8\}}$ and $J_{H \setminus \{11\}}$ are unmixed. Let $G'$ and $H'$ be the connected components of $G \setminus \{8\}$ and of $H \setminus \{11\}$ that are not a single vertex, i.e, $G' = G[\{1,2,3,4,5,6,7\}]$ and $H' = H[\{12,13,14,15,16\}]$. Then, we consider the graph $F$ obtained by gluing the graphs $G[V(G') \!\cup\! \{8\}]$ and $H[V(H') \!\cup\! \{11\}]$ identifying the vertices $8$ and $11$, see Figure \ref{F.2bChordalTraceable}.

One can check that $J_F$ is unmixed and we claim that it is also strongly unmixed. In fact, if we consider the cut vertex $v$ obtained by the identification of $8$ and $11$, then the graphs $F \setminus \{v\}$, $F_v$ and $F_v \setminus \{v\}$ are chordal and it can be shown that they are all accessible. Hence, by Theorem \ref{T.FinalChordal}, their binomial edge ideals are strongly unmixed. We conclude that $J_F$ is Cohen-Macaulay by Theorem \ref{T.StronglyUnmixedImpliesCM}.

The same holds if $F$ is the graph obtained by gluing $G[V(G') \cup \{i\}]$ and $H[V(H') \cup \{j\}]$ identifying the vertices $i$ and $j$, where $i \in \{2,5,8\}$ and $j \in \{11,15\}$.

Notice that $G$ is chordal and $H$ is traceable, but the resulting graph $F$ is not chordal nor traceable.
\end{example}

Example \ref{E.generalized2b} can be generalized as follows.

\begin{problem}
Let $G$ and $H$ be connected graphs, $v$ be a cut vertex of $G$ and $w$ a cut vertex of $H$. Set $G \setminus \{v\} = G_1 \sqcup G_2$, $H \setminus \{w\} = H_1 \sqcup H_2$ and suppose that $J_G$, $J_H$, $J_{G \setminus \{v\}}$ and $J_{H \setminus \{w\}}$ are unmixed. Let $F_{ij}$ be the graph obtained by gluing $G[V(G_i) \cup \{v\}]$ and $H[V(H_j) \cup \{w\}]$ identifying $v$ and $w$, for $i,j = 1,2$. If $J_G$ and $J_H$ are Cohen-Macaulay, is it true that $J_{F_{ij}}$ is Cohen-Macaulay? If $G$ and $H$ are accessible, is it true that $F_{ij}$ is accessible?
\end{problem}

\medskip

In \cite[Corollary 6.2]{BMS18}, we proved that for bipartite graphs, binomial edge ideals are the same up to isomorphism as Lovász–Saks–Schrijver ideals in two sets of variables (see \cite{HMMW15}), permanental edge ideals (see \cite[Section 3]{HMMW15}) and parity binomial edge ideals (see \cite{KSW16}), but this does not hold for non-bipartite graphs. Hence, even though Conjecture \ref{C.mainConj} would prove the field-independence of Cohen-Macaulayness for binomial edge ideals, this would not ensure the same for the other three classes. Indeed, Cohen-Macaulayness of permanental edge ideals depends on the field, as the following example shows.

\begin{example}
Recall that the permanental edge ideal of a graph $G$ on the vertex set $[n]$ is the ideal
\[
\Pi_G = (x_iy_j + x_jy_i : \{i,j\} \in E(G)) \subseteq K[x_1,\dots,x_n,y_1,\dots,y_n].
\]
Let $\mc K_4$ be the complete graph on $4$ vertices. If $\chara(K)=2$, then $\Pi_{\mc K_4} = J_{\mc K_4}$ is Cohen-Macaulay (since it is the ideal of $2$-minors of a $(2\times 4)$-generic matrix), whereas using \textit{Macaulay2} \cite{M2} one can see that $\Pi_{\mc K_4}$ is not Cohen-Macaulay if $K = \mathbb Q$. This shows that the Cohen-Macaulayness of permanental edge ideals cannot be characterized combinatorially.
\end{example}

Hence, it is natural to ask:

\begin{problem}
Does the Cohen-Macaulayness of Lovász–Saks–Schrijver ideals and of parity binomial edge ideals depend on the field? If not, is there a combinatorial description of Cohen-Macaulayness in terms of the underlying graph?
\end{problem}

\begin{example}
The graph $G$ in Figure \ref{F.LovaszFieldDependent} has the property that both the regularity and the projective dimension (and, hence, the depth) of the associated Lovász–Saks–Schrijver ideal $L_G$ in two sets of variables and of the parity binomial edge ideal $\mc I_G$ depend on the field. More precisely, if $R \! =\! K[x_i,y_i \!:\! i \in [8]]$, using \textit{Macaulay2} \cite{M2} one can see that:
\[
\pd(R/L_G) = \pd(R/\mc I_G) =
\begin{cases}
12 & \text{if } K = \mathbb Z_2\\
11 & \text{if } K = \mathbb Z_3
\end{cases}
\quad \text{and} \quad
\reg(R/L_G) = \reg(R/\mc I_G) =
\begin{cases}
7 & \text{if } K = \mathbb Z_2\\
6 & \text{if } K = \mathbb Z_3
\end{cases}.
\]
\end{example}

%\vspace*{-5mm}
\begin{figure}[ht!]
\begin{subfigure}[c]{0.45\textwidth}
\centering
\begin{tikzpicture}
\node[label={below:{\small $2$}}] (a) at (0,0) {};
\node[label={below:{\small $4$}}] (b) at (1,0) {};
\node[label={below:{\small $6$}}] (c) at (2,0) {};
\node[label={below:{\small $8$}}] (d) at (3,0) {};
\node[label={above:{\small $1$}}] (e) at (0,1.5) {};
\node[label={above:{\small $3$}}] (f) at (1,1.5) {};
\node[label={above:{\small $5$}}] (g) at (2,1.5) {};
\node[label={above:{\small $7$}}] (h) at (3,1.5) {};
\node[label={above:{\small $9$}}] (i) at (4,1.5) {};
\draw (0,1.5) -- (0,0) -- (1,0) -- (2,1.5) -- (1,1.5) -- (0,1.5) -- (2,0) -- (3,1.5) -- (3,0) -- (4,1.5)
(0,0) -- (3,1.5)
(3,0) -- (1,1.5)
(1,0) -- (4,1.5) -- (2,0) -- (2,1.5);
\end{tikzpicture}
\caption{A traceable graph $G$} \label{F.LovaszFieldDependent}
\end{subfigure}
\begin{subfigure}[c]{0.45\textwidth}
\centering
\begin{tikzpicture}
\node[label={above:$2$}] (a) at (0,1.5) {};
\node[label={above:$4$}] (b) at (1,1.5) {};
\node[label={above:$6$}] (c) at (2,1.5) {};
\node[label={above:$8$}] (d) at (3,1.5) {};
\node[label={below:$1$}] (e) at (0,0) {};
\node[label={below:$3$}] (f) at (1,0) {};
\node[label={below:$5$}] (g) at (2,0) {};
\node[label={below:$7$}] (h) at (3,0) {};
\draw (0,0) -- (0,1.5) -- (1,0) -- (1,1.5) -- (3,0) -- (3,1.5) -- (2,0) -- (2,1.5) -- (0,0) -- (1,1.5)
(0,1.5) -- (2,0)
(1,0) -- (3,1.5)
(2,1.5) -- (3,0);
\end{tikzpicture}
\caption{A bipartite graph $H$} \label{F.BipartiteFieldDependent}
\end{subfigure}
\vspace*{-3mm}
\caption{}
\end{figure}

As for binomial edge ideals, we do not know whether the depth or the extremal Betti numbers are independent of the field. However, their Betti numbers may be field-dependent.

\begin{example} \label{E.BipartiteFieldDependent}
Let $H$ be the bipartite graph in Figure \ref{F.BipartiteFieldDependent}. Using \textit{Macaulay2} \cite{M2}, one can check that some graded Betti numbers of $J_H$ are different over $\mathbb Z_2$ and over $\mathbb Z_3$.
\end{example}

\medskip

In \cite{JK19}, Jayanthan and Kumar compute the regularity of Cohen-Macaulay binomial edge ideals of bipartite graphs using the explicit description of these graphs given in \cite[Theorem 6.1 (c)]{BMS18}. By the proof of Corollary \ref{C.bipartiteCM}, these graphs are traceable. Thus, we ask the following:

\begin{problem}
Is it possible to find a formula or bounds for the regularity of Cohen-Macaulay binomial edge ideals of traceable graphs?
\end{problem}

\section*{Acknowledgements}
The authors acknowledge the extensive use of the software \textit{Macaulay2} \cite{M2} and \textit{Nauty} \cite{Nauty}. The first author was supported by INdAM. The second author was supported by the Einstein Foundation Berlin under Francisco Santos grant EVF-2015-23 and by the Deutsche Forschungsgemeinschaft (DFG, German Research Foundation) – project number 454595616. The authors thank the referees for their careful reading of the paper.

\section*{Data Availability}
Data sharing not applicable to this article as no datasets were generated or analysed during the current study.

\end{document}